\documentclass[12pt]{article}
\usepackage[cp1251]{inputenc}
\usepackage[english]{babel}
\usepackage{amssymb}
\usepackage{amsmath}
\usepackage{amsfonts}
\usepackage{amsthm}
\usepackage{xcolor}
\usepackage{graphicx}
\newtheorem{theorem}{Theorem}[section]
\newtheorem{definition}{Definition}[section]

\newtheorem{corollary}{Corollary}[section]
\newtheorem{remark}{Remark}[section]
\newtheorem{tverd}{Proposition}[section]

\begin{document}

\begin{center}
 \textbf{ASYMPTOTIC ANALYSIS OF A BOUNDARY-VALUE PROBLEM\\[2mm] IN A THIN CASCADE DOMAIN WITH A LOCAL JOINT}
\end{center}

\begin{center}
A. V. Klevtsovskiy, \ \ \ T.~A.~Mel'nyk
\end{center}
\begin{center}
{\small
Faculty of Mathematics and Mechanics, Department of Mathematical Physics\\
Taras Shevchenko National University of Kyiv\\
Volodymyrska str. 64,\ 01601 Kyiv,  \ Ukraine
\\
E-mail: avklevtsovskiy@gmail.com, \ \ melnyk@imath.kiev.ua}
\end{center}

\bigskip

\begin{abstract}
A nonuniform Neumann boundary-value problem is considered for the Poisson equation
 in a thin domain $\Omega_\varepsilon$ coinciding with two thin rectangles connected through a joint of diameter ${\cal O}(\varepsilon)$.
 A rigorous procedure is developed to construct the complete asymptotic expansion for the solution as the small parameter $\varepsilon \to 0.$ Energetic and  uniform pointwise estimates for the difference between the solution of the starting problem $(\varepsilon >0)$ and the solution of the corresponding limit problem  $(\varepsilon =0)$ are proved, from which the influence of the geometric irregularity of the joint is observed.
 \end{abstract}

\bigskip


 \qquad \qquad \ \qquad with a local geometrical irregularity

\medskip

{\bf MOS subject classification:} \  35C20, 35B40, 35J05, 74K30

\medskip


\section*{Introduction}

Investigations of various physical and biological processes in channels are urgent for numerous fields of natural sciences. Special interest of researchers is focused on various effects observed in vicinities of local irregularities of the geometry (widening or narrowing) of channels (e.g., adhesion to the walls, welds, and stenosis). Results of recent theoretical, experimental and numerical studies of flows and wall-pressure fluctuations in channels with different types of narrowing are summarized in \cite{Borysiuk2007-1,Borysiuk2007-2,Borysiuk2010} and references therein.
Also the study of influence of local geometrical irregularities is very important in engineering, since such irregularities often
directly affect the strength (stability, resistance, power, etc.) of constructions and devices. A fairly complete review on this topic has been presented in the article by Gaudiello and Kolpakov \cite{Gaudiello-Kolpakov}.

In addition, the paper \cite{Gaudiello-Kolpakov} is a pioneering paper, where
the influence of a local geometrical irregularity in a thin domain was studied  with the help of multi-scale approach.
Videlicet, the authors derived the limit problem for a homogeneous Neumann problem for the Poisson equation with a right-hand side that depends only of one longitudinal variable in junctions of thin domains and showed that the local geometric irregularity in the joint does not affect the view of the corresponding limit problem. However, the convergence theorem and asymptotic estimates have not been proved.

It should be stressed that the error estimates and convergence rate are very important both for the justification of  the adequacy of one- or two-dimensional models aimed at the description of actual three-dimensional thin bodies and for the study of boundary effects and effects of local (internal) inhomogeneities in applied problems. Particular importance for engineering practice is pointwise estimates for approximations, since large values of tearing stresses in small region at first involve local material damage and then the destruction of whole construction.
Those estimates can be obtained and substantiated as a result of the development of new asymptotic methods.

In \cite{Mel_Klev_NO} we proved the error estimates and constructed the asymptotic expansion for the solution of  a boundary-value problem in a thin cascade domain without joints (see Fig.~\ref{fig1}).
\begin{figure}[htbp]
\begin{center}
\includegraphics[width=14cm]{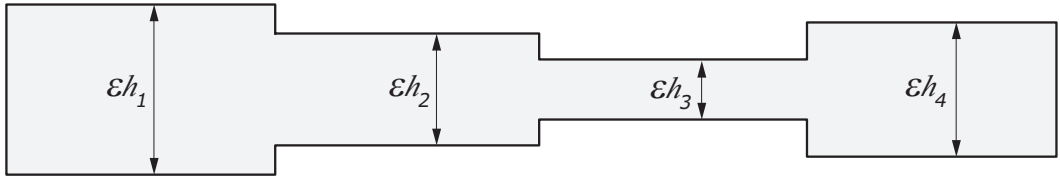}
\caption{Thin cascade domain without local joints}\label{fig1}
\end{center}
\end{figure}

The present paper is devoted to further development of the asymptotic method proposed in \cite{Mel_Klev_NO}.
 Namely, we consider  a nonuniform Neumann boundary-value problem for the Poisson equation with the right-hand side that depends both on longitudinal and transversal variables in a thin cascade domain that consists of two thin rectangles of different thicknesses and local
geometric irregularity (the joint) between them (this can be either a local widening (see. Fig.~\ref{fig2}) or local narrowing (see  Fig.~\ref{fig3})).
As we will see, there is no essential difference between the construction of the asymptotic expansion for the solution to
a boundary-value problem in  $2$- or $3$- dimensional thin cascade domain formed by two thin rods.
Therefore, to simplify calculations, we consider the two-dimensional case.

\begin{figure}[htbp]
\begin{center}
\includegraphics[width=14cm]{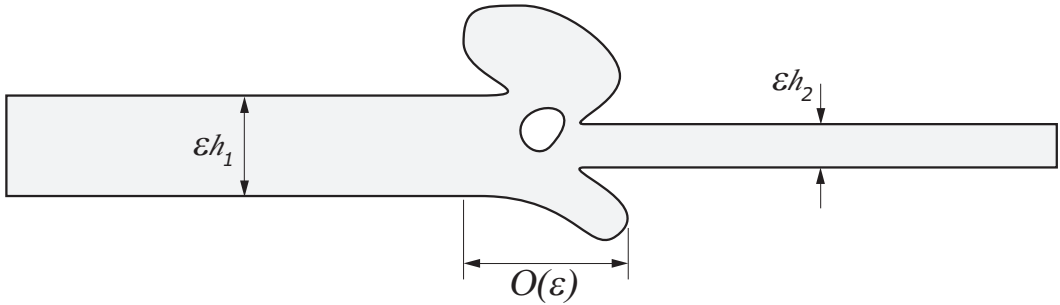}
\caption{Thin cascade domain with a local widening }\label{fig2}
\end{center}
\end{figure}

\begin{figure}[htbp]
\begin{center}
\includegraphics[width=14cm]{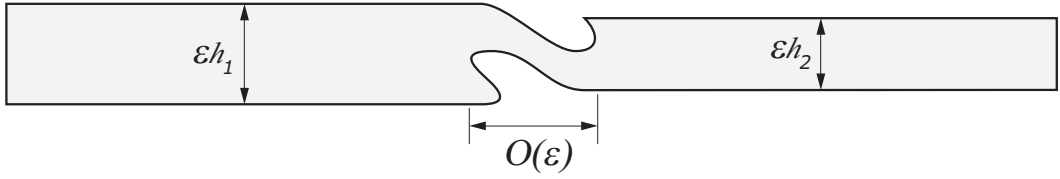}
\caption{Thin cascade domain with a local narrowing }\label{fig3}
\end{center}
\end{figure}

A principal new feature of this paper in comparison with the paper \cite{Gaudiello-Kolpakov} is the construction and justification
of the complete asymptotic expansion for the solution and the proof both energetic and  pointwise uniform estimates for the difference between the solution of the starting problem $(\varepsilon >0)$ and the solution of the corresponding limit problem  $(\varepsilon =0).$
As a result, it became possible to identify more precisely the impact of the geometric irregularity and material characteristics of the joint on
some properties of the whole structure.  In addition, on the one hand, our results confirm and complement some of conclusions of the article \cite{Gaudiello-Kolpakov}, on the other hand show that the main assumptions made in this article are not correct. A more informative discussion is given in the last section of the present paper.

The paper is organized as follows. In Section 2, we construct the formal asymptotic expansion for the solution to the problem (\ref{probl}). To perform this  we generalize the asymptotic method for
 boundary-value problems in thin domains with constant thickness proposed in monograph~ \cite{N-book-02}. In particular, we introduced a special inner asymptotic expansion in a neighborhood of the joint,  determine its coefficients and study some their properties as solutions to corresponding boundary-value problems in an unbounded domain. Thus, the asymptotics for the solution consists of three parts: the regular part, the boundary parts near the extreme vertical sides and the inner part in a neighborhood of the joint.

In Section 3, we justify the asymptotics (Theorem \ref{mainTheorem}) and prove asymptotic estimates for the leading terms of the asymptotics (Corollary \ref{Corollary}). In Conclusions  we analyze results obtained in this paper and discuss possible generalizations.

\section{Statement of the problem}

The model thin cascade domain $\Omega_\varepsilon$  consists of two thin rectangles
$$
\Omega_\varepsilon^{(1)}=\Big((-1,-\varepsilon\frac{l}{2})\times\Upsilon_\varepsilon^{(1)}\Big)\quad \text{and} \quad \Omega_\varepsilon^{(2)}=\Big((\varepsilon\frac{l}{2},1)\times\Upsilon_\varepsilon^{(2)}\Big)
$$
that are joined through $\Omega_\varepsilon^{(0)}$ (referred in the sequel "joint").
Here $\Upsilon_\varepsilon^{(i)}=\left(-\varepsilon\frac{h_i}{2},\varepsilon\frac{h_i}{2}\right), \ i=1,2;$ \ $\varepsilon$ is a small parameter; $l$, $h_1$ and $h_2$ are fixed positive constants.

The joint $\Omega_\varepsilon^{(0)}$ are formed by the homothetic transformation with the coefficient $\varepsilon$ from a domain
$\Xi^{(0)}$,  i.e., $\Omega_\varepsilon^{(0)} = \varepsilon\, \Xi^{(0)}.$
In addition, we assume that
$$
\Omega_\varepsilon^{(0)} \bigcap \big\{(x,y): \ |y| \le \varepsilon\, \max\{h_1, h_2\} \big\} \subset
\big\{(x,y): \ |x| \le (- \varepsilon \frac{l}{2} , \varepsilon \frac{l}{2} ) \big\}
$$
and  the interior of the union
$$
\overline{\Omega_\varepsilon^{(1)}} \ \cup \
\overline{\Omega_\varepsilon^{(0)}} \ \cup \ \overline{\Omega_\varepsilon^{(2)}}.
$$
is a domain with the Lipschitz boundary, which we denote by $\Omega_\varepsilon$ (see e.g. Fig. \ref{fig4}).

\begin{figure}[htbp]
\begin{center}
\includegraphics[width=14cm]{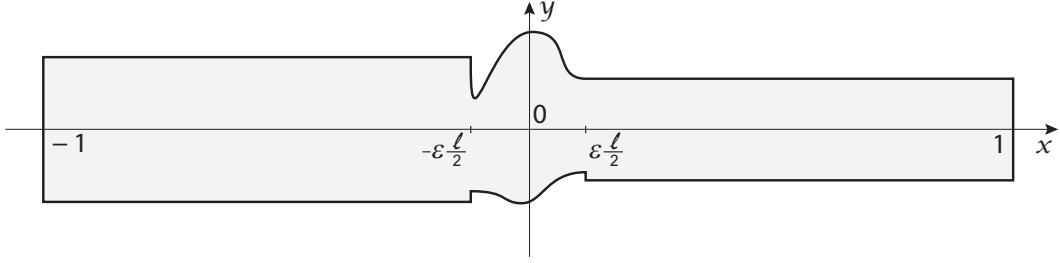}
\caption{The model thin cascade domain $\Omega_\varepsilon$ with a local geometric irregularity}\label{fig4}
\end{center}
\end{figure}

\begin{remark}
As an example of the joint we can consider the following domain:
$$
\Omega_\varepsilon^{(0)} = \varepsilon\,\Xi^{(0)}= \big\{ (x, y) : \ -\varepsilon h_0^-(\tfrac{x}{\varepsilon}) < y < \varepsilon h_0^+(\tfrac{x}{\varepsilon}), \ y\in(-\varepsilon\tfrac{l}2, \ \varepsilon\tfrac{l}2) \big\},
$$
where $\Xi^{(0)} = \big\{ (\xi,\eta) : \ -h_0^-(\xi) < \eta < h_0^+(\xi), \ \xi\in(-\tfrac{l}2, \ \tfrac{l}2) \big\},$
the functions $h_0^-$ and $h_0^+$ belong to the space $C^1 ([-\tfrac{l}2, \tfrac{l}2]),$  take positive values on the segment $[-\tfrac{l}2, \tfrac{l}2]$ and $h_0^\pm(-\tfrac{l}2) \le h_1,$  $h_0^\pm(\tfrac{l}2) \le h_2.$
With the help of functions $h_0^-$ and $h_0^+$ it is possible to describe both a local
narrowing and a local widening.
\end{remark}

In the domain $\Omega_\varepsilon,$ we consider the following mixed boundary-value problem:
\begin{equation}\label{probl}
\left\{\begin{array}{rcll}
-\Delta{u_\varepsilon}(x,y)                                                & = & f(x,\frac{y}{\varepsilon}),        & (x,y)\in\Omega_\varepsilon,                  \\[2mm]
-\partial_\nu{u_\varepsilon}(x,y)|_{y=\pm{\varepsilon\frac{h_i}{2}}}         & = & {\varepsilon\varphi_\pm^{(i)}}(x), & x\in{I_\varepsilon^{(i)}},\ i=1,2,           \\[2mm]
u_\varepsilon(-1,y)                                                        & = &
0,                                 & y\in\Upsilon_\varepsilon^{(1)},              \\[2mm]
u_\varepsilon(1,y)                                                         & = &
0,                                 & y\in\Upsilon_\varepsilon^{(2)},              \\[2mm]
\partial_\nu{u_\varepsilon}(x,y)                                           & = &
0,                                 & (x,y)\in\Gamma_\varepsilon,
\end{array}\right.
\end{equation}
\\[2mm]
where $I_\varepsilon^{(1)}=(-1,-\varepsilon\frac{l}{2}), \, I_\varepsilon^{(2)}=(\ \varepsilon\frac{l}{2},\ 1)$, \, $\partial_\nu$ is the outward normal derivative, and the boundary of the joint is described by the formula
$$
\Gamma_\varepsilon = \partial\Omega_\varepsilon \backslash \Big(\bigcup\limits_{i=1}^2 \left(\big(I_\varepsilon^{(i)} \times \{\pm\varepsilon\tfrac{h_i}2 \} \big) \cup \big(\{(-1)^i\}\times\Upsilon_\varepsilon^{(i)}\big) \right)\Big).
$$

Assume that the given functions $f$ and $\{\varphi_\pm^{(i)}\}$ are
smooth in the corresponding domains of definition.

It follows from the theory of linear boundary-value problems that, for any fixed value of $\varepsilon,$ problem (\ref{probl})
possesses a unique weak solution $u_\varepsilon$ from the Sobolev space $H^1(\Omega_\varepsilon)$ such that its traces on the vertical
end sides of the domain $\Omega_\varepsilon$ are equal to zero, i.e., $u_\varepsilon|_{x=\pm1}=0,$ and the solution satisfies the integral identity
\begin{equation}\label{int-identity}
\int_{\Omega_\varepsilon} \nabla u_\varepsilon \cdot \nabla \psi \, dx\,dy =\int_{\Omega_\varepsilon} f \,\psi \, dx\,dy
 \mp \varepsilon  \sum_{i=1}^2 \int_{I_\varepsilon^{(i)}} \varphi^{(i)}_{\pm} \, \psi\, dx
\end{equation}
for any function $\psi\in H^1(\Omega_\varepsilon)$ such that $\psi|_{x=\pm1}=0.$

\begin{remark}
In the right-hand side of identity (\ref{int-identity}), we introduce the abridged notation
$$
\mp \varepsilon  \sum_{i=1}^2 \int_{I_\varepsilon^{(i)}} \varphi^{(i)}_{\pm} \, \psi\, dx
:= - \varepsilon  \sum_{i=1}^2 \int_{I_\varepsilon^{(i)}} \varphi^{(i)}_{+} \, \psi\, dx
+ \varepsilon  \sum_{i=1}^2 \int_{I_\varepsilon^{(i)}} \varphi^{(i)}_{-} \, \psi\, dx
$$
and use it in what follows.
\end{remark}

{\sf
The aim of the present paper is to construct and justify the asymptotic expansion of the
solution $u_\varepsilon$ as $\varepsilon \to 0$.
}

\section{Formal asymptotic expansion}

\subsection{Regular part of the asymptotics}

We seek the regular part of the asymptotics in the form
\begin{equation}\label{regul}
u_\infty^{(i)}:=\sum\limits_{k=2}^{+\infty}\varepsilon^{k}\left(u_k^{(i)}(x,\frac{y}{\varepsilon})+\varepsilon^{-2}\omega_k^{(i)}(x)\right), \quad (x,y)\in\Omega^{(i)}_\varepsilon, \quad i=1,2.
\end{equation}
Formally substituting the series (\ref{regul}) into the differential equation and into the first boundary condition of problem (\ref{probl}), we obtain:
$$
-\sum\limits_{k=2}^{+\infty}\varepsilon^{k}\partial^2_{xx}{u}_k^{(i)}(x,\eta) -\sum\limits_{k=2}^{+\infty}\varepsilon^{k-2}\partial^2_{\eta\eta}{u}_k^{(i)}(x,\eta) -\sum\limits_{k=2}^{+\infty}\varepsilon^{k-2}\frac{d^2\omega_k^{(i)}}{dx^2}(x) \approx f(x,\eta), \quad \eta=\frac{y}{\varepsilon},
$$
$$
-\sum\limits_{k=2}^{+\infty}\varepsilon^{k}\partial_\eta{u}_k^{(i)}(x,\eta) \Big|_{\eta=\pm\frac{h_i}{2}} \approx \varepsilon^2\varphi_\pm^{(i)}(x),
$$
where $\partial_x=\partial/\partial{x},$ \,  $\partial^2_{xx}=\partial^2/\partial{x}^2,$ \,
$\partial_\eta=\partial/\partial{\eta},$ \,  $\partial^2_{\eta\eta}=\partial^2/\partial{\eta}^2.$

Equating the coefficients of the same powers of $\varepsilon$, we deduce recurrent relations of the boundary-value problems
for the determination of the expansion coefficients in (\ref{regul}). Let us consider the problem for  $u_2^{(i)}:$
\begin{equation}\label{regul_probl_2}
\left\{\begin{array}{rcl}
-\partial_{\eta\eta}^2{u}_2^{(i)}(x,\eta)          & = & f(x,\eta)+\dfrac{d^{\,2}\omega_2^{(i)}}{dx^2}(x), \quad \eta\in\Upsilon_i,\\[2mm]
-\partial_{\eta}u_2^{(i)}(x,\eta)|_{\eta=\pm\frac{h_i}{2}} & = & \varphi_\pm^{(i)}(x), \quad x\in I_\varepsilon^{(i)} \\[2mm]
\langle u_2^{(i)}(x,\cdot) \rangle_{\Upsilon_i}   & = & 0, \quad x\in I_\varepsilon^{(i)}.
\end{array}\right.
\end{equation}
Here $\Upsilon_i=\left(-\frac{h_i}{2},\frac{h_i}{2}\right), $ \
$\langle u(x,\cdot) \rangle_{\Upsilon_i} :=  \int_{\Upsilon_i}u (x,\eta)d\eta,$ $\ i=1,2.$

For each value of $i,$ the problem (\ref{regul_probl_2}) is the Neumann problem for the ordinary differential equation with respect to the variable $\eta\in\Upsilon_i$; here, the variable $x$ is regarded as a parameter. We now write the necessary and sufficient conditions for
the solvability of problem (\ref{regul_probl_2}) and obtain the following differential equation for the function $\omega_2^{(i)}:$
\begin{equation}\label{omega_probl_2}
- h_i\frac{d^2\omega_2^{(i)}}{dx^2}(x) = \int_{\Upsilon_i}f(x,\eta)d\eta - \varphi_+^{(i)}(x) + \varphi_-^{(i)}(x), \quad x\in I_\varepsilon^{(i)} \quad (i=1, 2).
\end{equation}
Let $\omega_2^{(i)}$ be a solution of the differential equation (\ref{omega_probl_2}) (boundary conditions for this differential equation will be determined later). Then the solution of problem (\ref{regul_probl_2}) exist and the third relation in (\ref{regul_probl_2}) supplies the uniqueness of solution.

For determination of the coefficients $u_3^{(i)}, \  i=1, 2,$ we obtain the following problems:
\begin{equation}\label{regul_probl_3}
\left\{\begin{array}{rcl}
-\partial_{\eta\eta}^2{u}_3^{(i)}(x,\eta)               & = & \dfrac{d^{\,2}\omega_3^{(i)}}{dx^2}(x) , \quad \eta\in\Upsilon_i, \\[2mm]
-\partial_{\eta}u_3^{(i)}(x,\eta)|_{\eta=\pm\frac{h_i}{2}} & = & 0,    \quad x\in I_\varepsilon^{(i)}                        \\[2mm]
\langle u_3^{(i)}(x,\cdot) \rangle_{\Upsilon_i}  & = & 0, \quad x\in I_\varepsilon^{(i)}.
\end{array}\right.
\end{equation}
Repeating the previous reasoning, we find  ${u}_3^{(i)}\equiv 0$  and
$\dfrac{d^{\,2}\omega_3^{(i)}}{dx^2}(x)=0$ $x\in I_\varepsilon^{(i)},$ $i=1, 2.$

Let us consider boundary-value problems for the functions  $u_k^{(i)}, \ k\geq 4, \ i=1, 2 :$
\begin{equation}\label{regul_probl_k}
\left\{\begin{array}{rcl}
-\partial_{\eta\eta}^2{u}_k^{(i)}(x,\eta)    & = & \dfrac{d^{\,2}\omega_k^{(i)}}{dx^2}(x) + \partial_{xx}^2{u}_{k-2}^{(i)}(x,\eta), \quad \eta\in\Upsilon_i, \\[2mm]
-\partial_{\eta}u_k^{(i)}(x,\eta)|_{\eta=\pm\frac{h_i}{2}} & = & 0,     \quad x\in I_\varepsilon^{(i)}                       \\[2mm]
\langle u_2^{(i)}(x,\cdot) \rangle_{\Upsilon_i}   & = & 0, \quad x\in I_\varepsilon^{(i)}.
\end{array}\right.
\end{equation}
Assume that all coefficients $u_2^{(i)},\dots,u_{k-1}^{(i)},\omega_2^{(i)},\dots,\omega_{k-1}^{(i)}$ of the expansion (\ref{regul}) are determined.
We find  $u_k^{(i)}$ and $\omega_{k}^{(i)}$ from problem (\ref{regul_probl_k}). It follows from the solvability condition of problem (\ref{regul_probl_k})  that
$$
h_i\frac{d^2\omega_k^{(i)}}{dx^2}(x)=-\int_{\Upsilon_i}\partial_x^2{u}_{k-2}(x,\eta)d\eta= -\partial_x^2\left(\int_{\Upsilon_i}u_{k-2}^{(i)}(x,\eta)d\eta\right)=0,
$$
i.e., $\omega^{(i)}_k$ is a linear function solving the differential equation
\begin{equation}\label{omega_probl_k}
\frac{d^2\omega_k^{(i)}}{dx^2}(x)=0, \quad x\in I_\varepsilon^{(i)}.
\end{equation}

\begin{remark}
Boundary conditions for the differential equations (\ref{omega_probl_2}) and (\ref{omega_probl_k}) are unknown in advance.
They will be determined in the process of construction of the asymptotics.
\end{remark}

Thus, the solution of problem (\ref{regul_probl_k}) is uniquely determined. Hence, the recursive procedure for the determination
of the coefficients of series (\ref{regul}) is uniquely solvable.

\begin{remark}
By using the recursive procedure for the boundary-value problem (\ref{regul_probl_k}), one can easily show that the functions $u^{(i)}_{2p+1}$ are identically equal to zero
for odd $k=2p+1, \, p\in\Bbb{N}.$
\end{remark}


 \subsection{Boundary asymptotics near the vertical sides of domain $\Omega_\varepsilon$}

In the previous section, we have considered the regular asymptotics taking into account the inhomogeneity of the right-hand side of the differential
equation in (\ref{probl}) and the boundary conditions on the horizontal sides of the thin domain $\Omega_\varepsilon.$ In what follows,
we construct the boundary part of the asymptotics compensating the residuals of the regular part of the asymptotics
at the left side of $\Omega_\varepsilon^{(1)}$ and the right one of $\Omega_\varepsilon^{(2)}$.

At the left vertical part of the boundary of $\Omega_\varepsilon^{(1)},$  we seek the boundary asymptotics  for the solution in the form
\begin{equation}\label{prim+}
\Pi_\infty^{(1)}:=\sum\limits_{k=0}^{+\infty}\varepsilon^{k}\Pi_k^{(1)}\left(\frac{1+x}{\varepsilon},\frac{y}{\varepsilon}\right).
\end{equation}
Substituting (\ref{prim+}) into (\ref{probl}) and collecting coefficients with the same powers of $\varepsilon$, we obtain the following mixed boundary-value problems:
\begin{equation}\label{prim+probl}
 \left\{\begin{array}{rll}
  -\Delta_{\xi\eta}\Pi_k^{(1)}(\xi,\eta)=                      & 0,                  & (\xi,\eta)\in(0,+\infty)\times\Upsilon_1,                                     \\[2mm]
  -\partial_\eta\Pi_k^{(1)}(\xi,\eta)|_{\eta=\pm\frac{h_1}{2}}=& 0,                  & \xi\in(0,+\infty),                                                            \\[2mm]
  \Pi_k^{(1)}(0,\eta)=                                         & \Phi_k^{(1)}(\eta), & \eta\in\Upsilon_1,                                                            \\[2mm]
  \Pi_k^{(1)}(\xi,\eta)\to                                     & 0,                  & \xi\to+\infty,\ \ \eta\in\Upsilon_1,
 \end{array}\right.
\end{equation}\\[3mm]
where
$$
\xi=\frac{1+x}{\varepsilon}, \quad \eta=\frac{y}{\varepsilon}, \quad
\Phi_k^{(1)} =-\omega_{k+2}^{(1)}(-1), \ k=0,1,
$$
$$
\Phi_k^{(1)}(\eta)=-u_k^{(1)}(-1,\eta)-\omega_{k+2}^{(1)}(-1), \quad k\geq 2.
$$

Using the method of separation of variables, we determine the solution
\begin{equation}\label{view_solution}
\Pi_k^{(1)}(\xi,\eta)=\sum\limits_{p=0}^{+\infty}\left[a_p^{(1)}e^{-\frac{2p\pi}{h_1}\xi}\cos\left(\frac{2p\pi}{h_1}\eta\right)+ b_p^{(1)}e^{-\frac{(2p+1)\pi}{h_1}\xi}\sin\left(\frac{(2p+1)\pi}{h_1}\eta\right)\right]
\end{equation}
of problem (\ref{prim+probl}) at a fixed index $k,$
where
$$
a_p^{(1)}=\frac{2}{h_1}\int\limits_{-\frac{h_1}{2}}^{\frac{h_1}{2}}\Phi_k^{(1)}(\eta)\cos\left(\frac{2p\pi}{h_1}\eta\right)d\eta, \quad b_p^{(1)}=\frac{2}{h_1}\int\limits_{-\frac{h_1}{2}}^{\frac{h_1}{2}}\Phi_k^{(1)}(\eta)\sin\left(\frac{(2p+1)\pi}{h_1}\eta\right)d\eta,
$$
$$
a_0^{(1)}=\frac{1}{h_1}\int\limits_{-\frac{h_1}{2}}^{\frac{h_1}{2}}\Phi_k^{(1)}(\eta)d\eta = \frac{1}{h_1}\int\limits_{-\frac{h_1}{2}}^{\frac{h_1}{2}}u_k^{(1)}(-1,\eta)d\eta-\omega_{k+2}^{(1)}(-1)=-\omega_{k+2}^{(1)}(-1).
$$
It follows from the fourth condition in (\ref{prim+probl})  that coefficient $a_0^{(1)}$
must be equal to $0.$ As a result,
we arrive at  the following boundary conditions for the functions $\{\omega_{k+2}^{(1)}\}$:
\begin{equation}\label{bv_left}
\omega_{k+2}^{(1)}(-1)=0, \quad k\in\Bbb{N}_0.
\end{equation}

At the left vertical part of the boundary of $\Omega_\varepsilon^{(2)},$  we seek the boundary asymptotics  for the solution in the form
\begin{equation}\label{prim-}
\Pi_\infty^{(2)}:=\sum\limits_{k=0}^{+\infty}\varepsilon^{k}\Pi_k^{(2)}\left(\frac{1-x}{\varepsilon},\frac{y}{\varepsilon}\right).
\end{equation}
We obtain the following problems for the determination of the coefficients $\{\Pi_k^{(2)}\}_{k\in\Bbb{N}_0}$:
\begin{equation}\label{prim-probl}
 \left\{\begin{array}{rll}
  -\Delta_{\xi^*\eta}\Pi_k^{(2)}(\xi^*,\eta)=                    & 0,                  & (\xi^*,\eta)\in(0,+\infty)\times\Upsilon_2,                                     \\[2mm]
  -\partial_\eta\Pi_k^{(2)}(\xi^*,\eta)|_{\eta=\pm\frac{h_2}{2}}=& 0,                  & \xi^*\in(0,+\infty),                                                           \\[2mm]
  \Pi_k^{(1)}(0,\eta)=                                           & \Phi_k^{(2)}(\eta), & \eta\in\Upsilon_2,                                                              \\[2mm]
  \Pi_k^{(1)}(\xi^*,\eta)\to                                     & 0,                  & \xi^*\to+\infty, \ \ \eta\in\Upsilon_2,
 \end{array}\right.
\end{equation}\\[3mm]
where
$$
\xi^*=\frac{1-x}{\varepsilon},\quad \eta=\frac{y}{\varepsilon}, \quad
\Phi_k^{(2)} =-\omega_{k+2}^{(2)}(1), \ \ k=0,1,
$$
$$
\Phi_k^{(2)}(\eta)=-u_k^{(2)}(1,\eta)-\omega_{k+2}^{(2)}(1), \ \  k\geq 2.
$$

Similarly we find the following solution of problem (\ref{prim-probl}) at a fixed index $k:$
\begin{equation}\label{view_solution2}
\Pi_k^{(2)}(\xi^*,\eta)=\sum\limits_{p=0}^{+\infty}\left[a_p^{(2)}e^{-\frac{2p\pi}{h_2}\xi^*}\cos\left(\frac{2p\pi}{h_2}\eta\right)+ b_p^{(2)}e^{-\frac{(2p+1)\pi}{h_2}\xi^*}\sin\left(\frac{(2p+1)\pi}{h_2}\eta\right)\right],
\end{equation}
where
$$
a_p^{(2)}=\frac{2}{h_2}\int\limits_{-\frac{h_2}{2}}^{\frac{h_2}{2}}\Phi_k^{(2)}(\eta)\cos\left(\frac{2p\pi}{h_2}\eta\right)d\eta, \quad b_p^{(2)}=\frac{2}{h_2}\int\limits_{-\frac{h_2}{2}}^{\frac{h_2}{2}}\Phi_k^{(2)}(\eta)\sin\left(\frac{(2p+1)\pi}{h_2}\eta\right)d\eta,
$$
$$
a_0^{(2)}=\frac{1}{h_2}\int\limits_{-\frac{h_2}{2}}^{\frac{h_2}{2}}\Phi_k^{(2)}(\eta)d\eta = \frac{1}{h_2}\int\limits_{-\frac{h_2}{2}}^{\frac{h_2}{2}}u_k^{(2)}(1,\eta)d\eta-\omega_{k+2}^{(2)}(1)=-\omega_{k+2}^{(2)}(1).
$$
It follows from the fourth condition in (\ref{prim-probl})  that the coefficient $a_0^{(2)}$ is equal to 0. This is possible if
\begin{equation}\label{bv_right}
\omega_{k+2}^{(2)}(1)=0, \quad k\in\Bbb{N}_0.
\end{equation}

\begin{remark}
Since $u_{k}^{(i)}\equiv0$ for $k=2p+1, \, p\in\Bbb{N}$, we conclude that $\Phi_k^{(i)}=0$
and, hence,
$$
\Pi_{0}^{(i)}\equiv0, \ \Pi_{2p-1}^{(i)}\equiv0, \quad p\in\Bbb{N}, \quad i=1,2.
$$
Moreover, from representation (\ref{view_solution}) and (\ref{view_solution2}) it follows the following asymptotic relations
\begin{equation}\label{as_estimates}
\begin{array}{c}
  \Pi_k^{(1)}(\xi,\eta) =  {\cal O}(\exp(- \tfrac{\pi}{h_1}\xi)) \quad \mbox{as} \quad \xi\to+\infty,
  \\[2mm]
 \Pi_k^{(2)}(\xi^*,\eta) = {\cal O}(\exp(- \tfrac{\pi}{h_2}\xi^*)) \quad \mbox{as} \quad \xi^*\to+\infty.
\end{array}
\end{equation}
\end{remark}

Equalities  (\ref{bv_left}) and (\ref{bv_right}) specify the boundary conditions at points $-1$ and $1$  for all functions $\{\omega_k^{(1)}\}$ and $\{\omega_k^{(2)}\}$, respectively.


 \subsection{Inner boundary part of the asymptotics}\label{inner_asymp}
 To obtain conditions for the functions $\{\omega_k^{(1)}\}$ and $\{\omega_k^{(2)}\}$ at the point $0,$ we introduce an additional internal asymptotics in a neighbourhood of the joint. For this we pass to the following variables $\xi=\frac{x}{\varepsilon}$ and $\eta=\frac{y}{\varepsilon}.$ Then forwarding the parameter $\varepsilon$ to 0, we see  that the domain $\Omega_\varepsilon$ is transformed into the unbounded domain $\Xi$, which is the union of joint~$\Xi^{(0)}$  and two half strips
$\Xi^{(1)}=(-\infty,-\tfrac{l}{2})\times\Upsilon_1,$ $\Xi^{(2)}=(\tfrac{l}{2},+\infty)\times\Upsilon_2,$
i.e., $\Xi$ is the interior of $\overline{\Xi^{(1)}\cup \Xi^{(0)}\cup \Xi^{(2)}}$.

Let us introduce the following notation for parts of the boundary of the domain $\Xi$:
\begin{itemize}
  \item
 $\partial\Xi^{(i)}_=$ is the horizontal parts of the boundary $\partial\Xi^{(i)},$ $i=1, 2,$
  \item
 $\Gamma = \partial\Xi \setminus \left( \partial\Xi^{(1)}_= \cup \partial\Xi^{(2)}_= \right)$.
\end{itemize}

We seek the inner expansion in the form
\begin{equation}\label{junc}
N_\infty=\sum\limits_{k=1}^{+\infty}\varepsilon^k N_k\left(\frac{x}{\varepsilon},\frac{y}{\varepsilon}\right).
\end{equation}
Substituting (\ref{junc}) into (\ref{probl}) and equating coefficients at the same powers of $\varepsilon$, we derive the following relations for $\{N_k\}:$
\begin{equation}\label{junc_probl_n}
 \left\{\begin{array}{rcll}
  -\Delta_{\xi\eta}{N_k}(\xi, \eta)                                  & =   & F_k(\xi,\eta), &
   \quad (\xi, \eta)\in\Xi,
   \\[2mm]
  \partial_\nu{N_k}(\xi, \eta)                             & =   & 0 &
   \quad (\xi, \eta)\in\Gamma,
   \\[2mm]
  -\partial_\eta{N_k}(\xi, \eta)|_{\eta=\pm \frac{h_i}{2}} & =   & B_{k_\pm}^{(i)}(\xi), &
   \quad (-1)^i \xi \in (\tfrac{l}{2}, +\infty), \ \ i=1, 2,
   \\[2mm]
   N_k(\xi, \eta)                & \sim & \omega^{(i)}_{k+2}(0) + \Psi^{(i)}_{k}(\xi, \eta), &
   \quad (-1)^i\xi \to +\infty, \ \eta \in \Upsilon_i, \ \ i=1, 2,
 \end{array}\right.
\end{equation}
where
$$
F_0\equiv F_1\equiv 0, \qquad
F_k(\xi,\eta) = \dfrac{\xi^{k-2}}{(k-2)!}\dfrac{\partial^{k-2}f}{\partial x^{k-2}}(0, \eta),
\quad (\xi, \eta)\in\Xi,
$$
$$
B_{1_\pm}^{(i)}\equiv B_{1_\pm}^{(i)} \equiv 0, \qquad
B_{k_\pm}^{(i)}(\xi) = \dfrac{\xi^{k-2}}{(k-2)!}\dfrac{d^{k-2}\varphi_\pm^{(i)}}{dx^{k-2}}(0),
\quad (-1)^i \xi \in (\tfrac{l}{2}, +\infty), \ \ i=1, 2,
$$
\begin{equation}\label{Psi_k}
\begin{array}{c}
\Psi_{0}^{(i)} \equiv 0, \qquad    \Psi_{1}^{(i)}(\xi, \eta) = \xi\dfrac{d\omega_{2}^{(i)}}{dx}(0),  \quad  i=1, 2,
  \\[2mm]
\Psi_{k}^{(i)}(\xi, \eta) = \xi\dfrac{d\omega_{k+1}^{(i)}}{dx}(0) +
    \dfrac{\xi^k}{k!}\dfrac{d^k\omega_2^{(i)}}{dx^k}(0) + \sum\limits_{j=0}^{k-2}
    \dfrac{\xi^j}{j!}\dfrac{\partial^j{u_{k-j}^{(i)}}}{\partial x^j}(0, \eta),
\quad  i=1, 2, \ \ k \geq 2.
\end{array}
\end{equation}

The right hand side and  boundary conditions for problem (\ref{junc_probl_n}) are obtained with the help of the Taylor decomposition of the functions $f$ and $\varphi_\pm^{(i)}$ at the point $x=0$.
The fourth condition in (\ref{junc_probl_n}) appears by matching the regular and inner asymptotics in a neighborhood of the joint, namely the asymptotics of
the terms $\{N_k\}$ as $\xi \to \pm \infty$ have to coincide with the corresponding asymptotics of  terms of the regular expansions (\ref{regul}) as $x \to \pm 0,$ respectively.
Expanding each term of the regular asymptotics in the Taylor series at the point $x=0$
and collecting the coefficients of the same powers of $\varepsilon$ with regard to (\ref{omega_probl_k}), we get relations (\ref{Psi_k}).

A solution of problem (\ref{junc_probl_n}) is sought in the form
\begin{equation}\label{new-solution}
N_k(\xi, \eta) = \Psi_{k}^{(1)}(\xi, \eta)\chi_1(\xi)
+ \Psi_{k}^{(2)}(\xi, \eta)\chi_2(\xi) + \widetilde{N}_k(\xi, \eta),
\end{equation}
where $ \chi_i \in C^{\infty}(\Bbb{R}_+),\ 0\leq \chi_i \leq1$ and
$$
\chi_i(\xi)=
\left\{\begin{array}{ll}
0, & \text{if} \ \ (-1)^i \xi \leq 1+\tfrac{l}{2},
\\[2mm]
1, & \text{if} \ \ (-1)^i \xi \geq 2+\tfrac{l}{2},
\end{array}\right. \quad i=1,2.
$$
Then $\widetilde{N}_k$ has to be a  solution of the following problem:
\begin{equation}\label{junc_probl_general}
 \left\{\begin{array}{rcll}
  -\Delta_{\xi\eta}{\widetilde{N}_k}(\xi, \eta)                                  & = &
   \widetilde{F}_k(\xi,\eta), &
   \quad (\xi, \eta)\in\Xi;
   \\[2mm]
  \partial_\nu{\widetilde{N}_k}(\xi, \eta)                             & = &
   0 &
   \quad (\xi, \eta)\in\Gamma;
   \\[2mm]
  -\partial_\eta{\widetilde{N}_k}(\xi, \eta)|_{\eta=\pm \frac{h_i}{2}} & = &
   \widetilde{B}_{k_{\pm}}^{(i)}(\xi), &
   \quad (-1)^i \xi \in (\tfrac{l}{2}, +\infty), \ \ i=1, 2,
   \\[2mm]
  \widetilde{N}_k(\xi, \eta)   & \to & \omega^{(i)}_{k+2}(0), &
   \quad (-1)^i\xi \to +\infty, \ \eta \in \Upsilon_i, \ \ i=1, 2,
 \end{array}\right.
\end{equation}
where $\widetilde{F}_0\equiv0,$
$$
\widetilde{F}_1(\xi,\eta) = \sum\limits_{i=1}^2 \Big(
\xi\dfrac{d\omega_2^{(i)}}{dx}(0) \chi_i^{\prime\prime}(\xi) +
2\dfrac{d\omega_2^{(i)}}{dx}(0) \chi_i^{\prime}(\xi) \Big),
$$
$$
\widetilde{F}_k(\xi, \eta) = \sum\limits_{i=1}^2 \Bigg[ \Big( \xi\dfrac{d\omega_{k+1}^{(i)}}{dx}(0) +
    \dfrac{\xi^k}{k!}\dfrac{d^k\omega_2^{(i)}}{dx^k}(0) + \sum\limits_{j=0}^{k-2}
    \dfrac{\xi^j}{j!}\dfrac{\partial^j{u_{k-j}^{(i)}}}{\partial x^j}(0, \eta) \Big)\chi_i^{\prime\prime}(\xi) 
$$
$$
+ 2\Big( \dfrac{d\omega_{k+1}^{(i)}}{dx}(0) +
    \dfrac{\xi^{k-1}}{(k-1)!}\dfrac{d^k\omega_2^{(i)}}{dx^k}(0) + \sum\limits_{j=1}^{k-2}
    \dfrac{\xi^{j-1}}{(j-1)!}\dfrac{\partial^j{u_{k-j}^{(i)}}}{\partial x^j}(0, \eta) \Big)\chi_i^{\prime}(\xi) 
$$
$$
- \dfrac{\xi^{k-2}}{(k-2)!}\dfrac{\partial^{k-2}f}{\partial x^{k-2}}(0, \eta)\chi_i(\xi) \Bigg] + \dfrac{\xi^{k-2}}{(k-2)!}\dfrac{\partial^{k-2}f}{\partial x^{k-2}}(0, \eta)
$$
and
$$
\widetilde{B}_{0_\pm}^{(i)} \equiv \widetilde{B}_{1_\pm}^{(i)} \equiv 0, \qquad
\widetilde{B}_{k_\pm}^{(i)}(\xi) = \dfrac{\xi^{k-2}}{(k-2)!}\dfrac{d^{k-2}\varphi_\pm^{(i)}}{dx^{k-2}}(0)
\big( 1-\chi_i(\xi) \big),
\quad  i=1, 2, \qquad k \geq 2.
$$

To study the solvability of problem (\ref{junc_probl_general}), we use the approach proposed in \cite{Naz96}.
 Let $C^{\infty}_{0,\xi}(\overline{\Xi})$ be a space of functions infinitely differentiable in $\overline{\Xi}$ and finite with
respect to  $\xi$, i.e.,
$$
\forall \,v\in C^{\infty}_{0,\xi}(\overline{\Xi}) \quad \exists \,R>0 \quad \forall \, (\xi,\eta)\in\overline{\Xi} \quad |\xi|\geq R: \quad v(\xi,\eta)=0.
$$
We now define a  space  $\mathcal{H}:=\overline{\left( C^{\infty}_{0,\xi}(\overline{\Xi}), \ \| \cdot \|_\mathcal{H} \right)}$, where
$$
\|v\|_\mathcal{H}=\sqrt{\int_\Xi|\nabla v(\xi, \eta)|^2 \, d\xi d\eta+ \int_\Xi |v(\xi, \eta)|^2 |\rho(\xi)|^2 \, d\xi d\eta} ,
$$
and the function  $\rho(\xi)= (1+|\xi|)^{-1}, \ \xi \in \Bbb R.$

\begin{definition}
A function $\widetilde{N}_k$ from the space $\mathcal{H}$ is called a weak solution of problem (\ref{junc_probl_general}) if the identity
\begin{equation}\label{integr}
\int\limits_{\Xi}\nabla \widetilde{N}_k \cdot \nabla v \, d\xi d\eta =
\int\limits_{\Xi} \widetilde{F}_k \, v \, d\xi d\eta \mp \int\limits^{-\frac{l}{2}}_{-\infty} \widetilde{B}^{(1)}_{k_\pm}(\xi) \, v(\xi, \pm \tfrac{h_1}{2})\, d\xi
\mp \int\limits_{\frac{l}{2}}^{+\infty} \widetilde{B}^{(2)}_{k_\pm}(\xi) \, v(\xi, \pm \tfrac{h_2}{2})\, d\xi.
\end{equation}
holds for all $v\in\mathcal{H}$.
\end{definition}

From lemma 4.1, remarks 4.1 and 4.2, corollary 4.1 (see \cite{ZAA99}) it follows the following propositions.
\begin{tverd}\label{tverd1}
   Let $\rho^{-1} \widetilde{F}_k\in L^2(\Xi),$ $\rho^{-1} \widetilde{B}^{(2)}_{k_\pm} \in L^2(\tfrac{l}{2}, +\infty)$ and $\rho^{-1} \widetilde{B}^{(1)}_{k_\pm} \in L^2(-\infty, -\tfrac{l}{2}).$

Then there exist a weak solution of problem (\ref{junc_probl_general}) if and only if
\begin{equation}\label{solvability}
\int_{\Xi} \widetilde{F}_k \, d\xi d\eta \mp \int^{-\frac{l}{2}}_{-\infty} \widetilde{B}^{(1)}_{k_\pm}(\xi) \,  d\xi
\mp \int_{\frac{l}{2}}^{+\infty} \widetilde{B}^{(2)}_{k_\pm}(\xi) \,  d\xi = 0 .
\end{equation}
This solution is defined up to an additive constant.
The additive constant  can be chosen to guarantee  the existence and uniqueness of a weak solution of problem (\ref{junc_probl_general}) with differentiable asymptotics
\begin{equation}\label{inner_asympt_general}
\widetilde{N}_k(\xi,\eta)=\left\{
\begin{array}{rl}
{\cal O}(\exp( \frac{\pi}{h_1}\xi)) & \mbox{as} \ \ \xi\to-\infty,
\\[2mm]
 \delta_k^+ + {\cal O}(\exp(-\frac{\pi}{h_2}\xi)) & \mbox{as} \ \ \xi\to+\infty.
\end{array}
\right.
\end{equation}
\end{tverd}
\begin{tverd}\label{tverd2}
The corresponding homogeneous problem for problem (\ref{junc_probl_general})
\begin{equation}\label{hom_probl}
  -\Delta_{\xi\eta}\mathfrak{N} = 0 \ \ \text{in} \ \ \Xi, \qquad
  \partial_\nu \mathfrak{N} = 0 \ \ \text{on} \ \ \partial \Xi,
\end{equation}
has a solution $\mathfrak{N}_0$ that does not belong to the space
$ {\cal H}$ and it has  the following differentiable asymptotics:
\begin{equation}\label{inner_asympt_hom_solution}
\mathfrak{N}_0(\xi,\eta)=\left\{
\begin{array}{rl}
\frac{1}{h_1} \, \xi + {\cal O}(\exp( \frac{\pi}{h_1}\xi)) & \mbox{as} \ \ \xi\to-\infty,
\\[2mm]
 C_0 +  \frac{1}{h_2} \, \xi + {\cal O}(\exp(-\frac{\pi}{h_2}\xi)) & \mbox{as} \ \ \xi\to+\infty.
\end{array}
\right.
\end{equation}
Any other solution to the homogeneous problem, which has polynomial growth at infinity, can be presented as a linear combination
$ \alpha_1 + \alpha_0 \mathfrak{N}_0.$
\end{tverd}
\begin{tverd}\label{tverd3}
 If the domain $\Xi$ is symmetric about the horizontal axis, the function $\widetilde{F}_k$ is even with respect to the variable $\eta$ $(\widetilde{F}_k$ is odd with respect to  $\eta)$ and
$\widetilde{B}^{(i)}_{k_-}\equiv - \widetilde{B}^{(i)}_{k_+}, \ i=1, 2$ $(\widetilde{B}^{(i)}_{k_-}\equiv \widetilde{B}^{(i)}_{k_+}, \ i=1, 2),$ then solution $\widetilde{ N}_k$ is an even (odd)
function with respect to $\eta.$ If $\widetilde{ N}_k$ is an odd function, then the constant  $\delta_k^+$ in (\ref{inner_asympt_general})  is equal to zero.
\end{tverd}

\begin{remark}\label{remark_constant}
Using the second Green-Ostrogradsky formula, similarly as was done in remark~4.3 (\cite{ZAA99}),
constant $\delta_k^+ \, (k\in\Bbb{N})$ in (\ref{inner_asympt_general}) can be found as follows
\begin{multline}\label{const_d_0}
\delta_k^+ = \frac{1}{h_2} \Bigg(
\int\limits_{\Xi} \xi \, \widetilde{F}_k(\xi, \eta) \, d\xi d\eta
\mp \int\limits^{-\frac{l}{2}}_{-\infty} \xi \, \widetilde{B}^{(1)}_{k_\pm}(\xi) \, d\xi
\mp \int\limits_{\frac{l}{2}}^{+\infty} \xi \widetilde{B}^{(2)}_{k_\pm}(\xi) \, d\xi \,-
\int\limits_\Gamma \partial_\nu(\xi)\, \widetilde{ N}_k (\xi,\eta) \ d\sigma_{\xi\eta} \Bigg).
\end{multline}
\end{remark}

It follows from Proposition \ref{tverd2} that problem (\ref{junc_probl_general}) at $k=0$ has a solution if and only if
\begin{equation}\label{trans0}
\omega_2^{(1)} (0) = \omega_2^{(2)} (0);
\end{equation}
in this case
\begin{equation}\label{N_0}
N_0 \equiv \widetilde{N}_0 \equiv \omega_2^{(1)} (0).
\end{equation}

Let us verify the solvability condition  (\ref{solvability}).
Taking into account  the  third relation in problems (\ref{regul_probl_2}) and (\ref{regul_probl_k}), the equality  (\ref{solvability}) can be re-written as follows:
$$
h_2 \frac{d\omega_{k}^{(2)}}{dx}(0) - h_1 \frac{d\omega_{k}^{(1)}}{dx}(0)
$$
$$
+ h_1 \int\limits_{-\frac{l}{2} -2}^{-\frac{l}{2} -1}
\frac{\xi^{k-1}}{(k-1)!}\frac{d^k\omega_2^{(1)}}{dx^k}(0) \, \chi_1^{\prime}(\xi) \,d\xi
+
h_2 \int\limits_{\frac{l}{2} +1}^{\frac{l}{2} +2}
\frac{\xi^{k-1}}{(k-1)!}\frac{d^k\omega_2^{(2)}}{dx^k}(0) \, \chi_2^{\prime}(\xi) \,d\xi
$$
$$
- \ h_1 \int\limits_{-\infty}^{-\frac{l}2} \frac{\xi^{k-2}}{(k-2)!} \dfrac{d^k\omega_2^{(1)}}{dx^k}(0) (1-\chi_1(\xi)) \ d\xi \ - \ h_2
\int\limits_{\frac{l}2}^{+\infty} \frac{\xi^{k-2}}{(k-2)!} \dfrac{d^k\omega_2^{(2)}}{dx^k}(0) (1-\chi_2(\xi)) \ d\xi
$$
$$
+ \ \int_{\Xi^{(0)}} \frac{\xi^{k-2}}{(k-2)!}\dfrac{\partial^{k-2}f}{\partial x^{k-2}}(0, \eta) \ d\xi d\eta \ = \ 0, \quad k\in\Bbb N, \ \ k \geq2.
$$
Whence, integrating by parts in the first two integrals with regard to (\ref{omega_probl_2}), we obtain the following relations for $\{\omega_k^{(i)}\}:$
\begin{equation}\label{transmisiont1}
h_2 \frac{d\omega_{k}^{(2)}}{dx}(0) - h_1 \frac{d\omega_{k}^{(1)}}{dx}(0) = -d_k^*, \quad k\in\Bbb N, \ \ k \geq2,
\end{equation}
where $d_2^* = 0,$
\begin{multline}\label{const_d_*}
d_k^* = \sum\limits_{i=1}^2 (-1)^{i+1}\frac{((-1)^{i}\tfrac{l}{2})^{k-2}}{(k-2)!}
\Big( \int_{\Upsilon_i}\dfrac{\partial^{k-3}f}{\partial x^{k-3}}(0, \eta) \ d\eta \mp
\dfrac{d^{k-3}\varphi_\pm^{(i)}}{dx^{k-3}}(0) \Big)
\\
+
\int_{\Xi^{(0)}} \frac{\xi^{k-3}}{(k-3)!}\dfrac{\partial^{k-3}f}{\partial x^{k-3}}(0, \eta) \ d\xi d\eta,
\quad k\in\Bbb N, \ \ k \geq3.
\end{multline}

Hence, if the functions $\omega_k^{(1)}$ and $\omega_k^{(2)}$ satisfy (\ref{transmisiont1}), then there exist a weak solution of the problem (\ref{junc_probl_general}). According to Proposition  \ref{tverd1}, it can be chosen in a unique way to guarantee the asymptotics (\ref{inner_asympt_general}).
However, we do not take into account the limit relations at infinity  in (\ref{junc_probl_general}) (see the forth condition). In order to satisfy them we add $\omega_{k+2}^{(1)}(0)$ to our solution (Proposition  \ref{tverd1} gives us that possibility) and derive the following conditions:
\begin{equation}\label{trans1}
 \omega_k^{(1)}(0) + \delta_{k-2}^+ = \omega_k^{(2)}(0) \ , \quad k\in\Bbb N, \ \ k \geq3.
\end{equation}
As a result, we get the solution of the problem (\ref{junc_probl_n}) with the following asymptotics:
\begin{equation}\label{inner_asympt}
{N}_{k}(\xi,\eta)=\left\{
\begin{array}{rl}
\omega_{k+2}^{(1)}(0) + \Psi_k^{(1)}(\xi, \eta) +
{\cal O}(\exp( \frac{\pi}{h_1}\xi)) & \mbox{as} \ \ \xi\to-\infty, \\[2mm]
\omega_{k+2}^{(2)}(0) + \Psi_k^{(2)}(\xi, \eta) +
{\cal O}(\exp(-\frac{\pi}{h_2}\xi)) & \mbox{as} \ \ \xi\to+\infty.
\end{array}
\right.
\end{equation}

Let us denote by
$$
G_k(\xi,\eta) :=
\left\{\begin{array}{rl}
\omega_{k+2}^{(1)}(0) + \Psi_k^{(1)}(\xi,\eta), & \xi < 0, \\[2mm]
\omega_{k+2}^{(2)}(0) + \Psi_k^{(2)}(\xi,\eta), & \xi > 0,
\end{array}\right.  \quad k\in \Bbb N.
$$
\begin{remark}\label{rem_exp-decrease}
Due to (\ref{inner_asympt}),
functions $\{{N}_{k} - G_k\}_{k\in \Bbb N}$ are exponentially decrease as $\xi \to \pm\infty.$
\end{remark}


 \subsection{Limit problem}

Relations (\ref{trans0}), (\ref{trans1}) together with (\ref{transmisiont1}), (\ref{omega_probl_2}), (\ref{omega_probl_k}), (\ref{bv_left}) and (\ref{bv_right}) complete boundary-value problems to
determine the functions $\{\omega_k^{(i)}\}.$

So for the functions $\omega_2^{(1)}$ and $\omega_2^{(2)}$ that form the main term of the regular asymptotic expansion (\ref{regul}), we obtain the following problem:
\begin{equation}\label{main}
\left\{\begin{array}{rcl}
- h_i\dfrac{d^2\omega_2^{(i)}}{dx^2}(x) & = &  \widehat{F}^{(i)}(x),  \qquad x\in I_i, \ \ i=1, 2,
\\[3mm]
\omega_2^{(1)}(0) & = & \omega_2^{(2)}(0),
\\[2mm]
h_1 \dfrac{d\omega_2^{(1)}}{dx}(0) & = & h_2 \dfrac{d\omega_2^{(2)}}{dx}(0),
\\[4mm]
\omega_2^{(1)}(-1) & = & 0, \quad \omega_2^{(2)}(1)  \ = \ 0 ,
\end{array}\right.
\end{equation}
where $I_1=(-1,0),$ $I_2=(0,1),$
\begin{equation}\label{right-hand-side}
\widehat{F}^{(i)}(x):= \int_{\Upsilon_i}f(x,\eta)\, d\eta - \varphi_+^{(i)}(x) + \varphi_-^{(i)}(x), \quad \ x\in I_i, \ \ i=1, 2.
\end{equation}
The problem (\ref{main}) is  called {\it limit problem} for problem (\ref{probl}).  The solution to (\ref{main}) is given by the following formulas:
\begin{multline}
\omega_2^{(1)}(x) = \frac{1}{h_1}\int\limits_{-1}^x (s-x)\widehat{F}^{(1)}(s)ds
\\
- \frac{(x+1)}{h_1+h_2}\Big( \int\limits_{-1}^0 \big( \tfrac{h_2}{h_1}s-1 \big)\widehat{F}^{(1)}(s)ds +
\int\limits_0^1 (1-s)\widehat{F}^{(2)}(s)ds \Big), \ \ x\in I_1;
\end{multline}
\begin{multline}
\omega_2^{(2)}(x) = \frac{1}{h_2}\int\limits_{x}^1 (s-x)\widehat{F}^{(2)}(s)ds
\\
-
\frac{(1-x)}{h_1+h_2}\Big(  \int\limits_{0}^1 \big( \tfrac{h_1}{h_2}s+1 \big)\widehat{F}^{(2)}(s)ds -
\int\limits_{-1}^0 (1+s)\widehat{F}^{(1)}(s)ds \Big), \ \ x\in I_2.
\end{multline}

For next functions  $\{\omega_k^{(1)}, \  \omega_k^{(2)}: \  k \geq 3\},$ the problems take the form
\begin{equation}\label{omega_probl*}
\left\{\begin{array}{rcll}
- h_i\dfrac{d^2\omega_k^{(i)}}{dx^2}(x) & = & 0,  &   x\in I_i,   \ \ i=1, 2,
\\[3mm]
\omega_k^{(1)}(0) & = & \omega_k^{(2)}(0) - \delta_{k-2}^+,
\\[2mm]
h_1\dfrac{d\omega_k^{(1)}}{dx}(0) & = & h_2\dfrac{d\omega_k^{(2)}}{dx}(0) + d_k^*,
\\[4mm]
\omega_k^{(1)}(-1) & = & 0, \quad \omega_k^{(2)}(1)  \ = \ 0 .
\end{array}\right.
\end{equation}
It is easy to verify that the solution to problem \eqref{omega_probl*} is given by the formulas
\begin{equation}\label{solutions}
\begin{array}{c}
\omega_k^{(1)}(x) =  \dfrac{(d_k^* - h_2 \, \delta_{k-2}^{+})}{h_1+h_2} \, (x + 1),  \ \ x\in I_1;
\\[5mm]
\omega_k^{(2)}(x) =  \dfrac{(d_k^* + h_1\, \delta_{k-2}^{+})}{h_1+h_2} \, (1 - x),  \ \ x\in I_2.
\end{array}
\end{equation}

\section{Complete asymptotic expansion and its justification}

From the limit problem (\ref{main}) we uniquely determine the first term of the asymptotics $\omega_2$ of  series (\ref{regul}). Next from the
equality (\ref{N_0}) we obtain the first term $N_0$ of the inner asymptotic expansion (\ref{junc}).
Then we rewrite problems (\ref{regul_probl_2}) in the form
\begin{equation}\label{new_regul_probl_2}
\left\{\begin{array}{rcl}
-\partial_{\eta\eta}^2{u}_2^{(i)}(x,\eta)          & = & f(x,\eta) - h_i^{-1} \widehat{F}^{(i)}(x), \quad \eta\in\Upsilon_i,\\[2mm]
-\partial_{\eta}u_2^{(i)}(x,\eta)|_{\eta=\pm\frac{h_i}{2}} & = & \varphi_\pm^{(i)}(x), \quad x\in I_i \\[2mm]
\langle u_2^{(i)}(x,\cdot) \rangle_{\Upsilon_i}   & = & 0, \quad x\in I_i,
\end{array}\right. \qquad i=1, 2,
\end{equation}
and find that
\begin{equation}\label{solution_t}
u_2^{(i)}(x,\eta)=-\int_{-\frac{h_i}{2}}^\eta(\eta-t)\Big(f(x,t)
- h_i^{-1} \widehat{F}^{(i)}(x) \Big)dt - \eta\varphi_-^{(i)}(x) + \alpha_2^{(i)}(x),
\end{equation}
where function  $\alpha_2^{(i)}$ are uniquely determined from third condition in (\ref{new_regul_probl_2}), i.e.
$$
\alpha_2^{(i)}(x) = \int_{\Upsilon_i} \int_{-\frac{h_i}{2}}^\eta(\eta-t) \, f(x,t)\, dt\, d\eta
-  6^{-1} h_i^2 \widehat{F}^{(i)}(x), \quad  i= 1, 2;
$$
functions $\widehat{F}^{(1)}$ and $\widehat{F}^{(2)}$  are defined by relations  (\ref{right-hand-side}).

Now with the help of formulas (\ref{view_solution}) and (\ref{view_solution2}), we determine the first terms $\Pi_2^{(1)}$ and $\Pi_2^{(2)}$ of the
boundary-asymptotic expansions (\ref{prim+})
and  (\ref{prim-}) respectively, as solutions of problems  (\ref{prim+probl}) and (\ref{prim-probl}) that can be rewritten as follows:
\begin{equation}\label{new_prim+probl}
 \left\{\begin{array}{rll}
  -\Delta_{\xi\eta}\Pi_2^{(1)}(\xi,\eta)=                      & 0,                  & (\xi,\eta)\in(0,+\infty)\times\Upsilon_1,                                     \\[2mm]
  -\partial_\eta\Pi_2^{(1)}(\xi,\eta)|_{\eta=\pm\frac{h_1}{2}}=& 0,                  & \xi\in(0,+\infty),                                                            \\[2mm]
  \Pi_2^{(1)}(0,\eta)=                                         & -u_2^{(1)}(-1,\eta), & \eta\in\Upsilon_1,                                                            \\[2mm]
  \Pi_2^{(1)}(\xi,\eta)\to                                     & 0,                  & \xi\to+\infty,\ \ \eta\in\Upsilon_1,
 \end{array}\right.
\end{equation}
\begin{equation}\label{new_prim-probl}
 \left\{\begin{array}{rll}
  -\Delta_{\xi^*\eta}\Pi_2^{(2)}(\xi^*,\eta)=                    & 0,                  & (\xi^*,\eta)\in(0,+\infty)\times\Upsilon_2,                                     \\[2mm]
  -\partial_\eta\Pi_2^{(2)}(\xi^*,\eta)|_{\eta=\pm\frac{h_2}{2}}=& 0,                  & \xi^*\in(0,+\infty),                                                           \\[2mm]
  \Pi_2^{(1)}(0,\eta)=                                           & -u_2^{(2)}(1,\eta), & \eta\in\Upsilon_2,                                                              \\[2mm]
  \Pi_2^{(1)}(\xi^*,\eta)\to                                     & 0,                  & \xi^*\to+\infty, \ \ \eta\in\Upsilon_2.
 \end{array}\right.
\end{equation}

The second term  ${N}_1$ of the inner asymptotic expansion (\ref{junc}) is the unique solution of the problem (\ref{junc_probl_n}) that can now be rewritten in the form
\begin{equation}\label{new_junc_probl_n}
 \left\{
 \begin{array}{rcll}
  -\Delta_{\xi\eta}{{N}_1}(\xi, \eta)                                  & =   & 0, &
   \quad (\xi, \eta)\in\Xi,
   \\[2mm]
  \partial_\nu{{N}_1}(\xi, \eta)                             & =   & 0, &
   \quad (\xi, \eta)\in\partial\Xi,
   \\[2mm]
   {N}_1(\xi, \eta) & \sim & \dfrac{d_3^*+(-1)^i h_{3-i}\delta_{1}^+}{h_1+h_2} + \xi\dfrac{d\omega_{2}^{(i)}}{dx}(0), & (-1)^i\xi \to +\infty, \ \eta \in \Upsilon_i, \ \ i=1, 2,
 \end{array}
 \right.
\end{equation}
with asymptotics (\ref{inner_asympt}). Recall that the constant $d_3^*$ is determined by formula (\ref{const_d_*})
and the constant  $\delta_1^+$ is also uniquely determined (see Remark~\ref{remark_constant}) by formula
\begin{equation}\label{delta_1}
\delta_1^+ = - \frac{1}{h_2} \left( \int_\Gamma \partial_\nu \xi \ \widetilde{N}_1 (\xi,\eta) \ d\sigma_{\xi\eta} \right).
\end{equation}

Thus we have uniquely determined the first terms of the asymptotic expansions (\ref{regul}), (\ref{prim+}),  (\ref{prim-}) and (\ref{junc}).

Assume that we have determined coefficients $\omega_2^{(i)},\ldots,\omega_{2n-2}^{(i)},$
$u_2^{(i)}, \, u_4^{(i)},\ldots,u_{2n-2}^{(i)}$ of the series (\ref{regul}),
coefficients
$\Pi^{(i)}_2,  \Pi^{(i)}_4,\ldots,\Pi^{(i)}_{2n-2}$ of the series (\ref{prim+}) and (\ref{prim-}) respectively,
coefficients
${N}_1,\ldots,{N}_{2n-3}$ of the series (\ref{junc}) and constants $\delta_1^+,\ldots,\delta_{2n-3}^+$.

Then, using formulas (\ref{solutions}),  we write the solution $\omega_{2n-1}$ of problem (\ref{omega_probl*})  with the constant  $\delta^+_{2n-3}$ in the first transmission condition.
It should be noted that constants $\{d_k^*\}_{k\geq3}$ depend only on $f$ and $\varphi_{\pm}^{(i)}, \ i=1,2 \ $ and they are uniquely defined  by formulas (\ref{const_d_*}).
Further we find the coefficient ${N}_{2n-2}$ of the inner asymptotic expansion (\ref{junc}),
which is the unique solution of the problem (\ref{junc_probl_n}) that can now be rewritten in the form
\begin{equation}\label{junc_probl_odd_2n-2}
 \left\{
 \begin{array}{l}
 \begin{array}{rcll}
  -\Delta_{\xi\eta}{}{{N}_{2n-2}}(\xi, \eta)                                  & =   &
   \dfrac{\xi^{2n-4}}{(2n-4)!}\dfrac{\partial^{2n-4}f}{\partial x^{2n-4}}(0, \eta), &
   \quad (\xi, \eta)\in\Xi,
   \\[4mm]
  \partial_\nu{{N}_{2n-2}}(\xi, \eta)                             & =   & 0, &
   \quad (\xi, \eta)\in\Gamma,
   \\
  -\partial_\eta{{N}_{2n-2}}(\xi, \eta)|_{\eta=\pm \frac{h_i}{2}} & =   &
   \dfrac{\xi^{2n-4}}{(2n-4)!}\dfrac{d^{2n-4}\varphi_\pm^{(i)}}{dx^{2n-4}}(0), &
   \quad (-1)^i \xi \in (\tfrac{l}{2}, +\infty), \ \ i=1, 2,
 \end{array}
 \\[4mm]
 \qquad \qquad \qquad \quad
  {N}_{2n-2} \ \ \sim \ \,
    \dfrac{d_{2n-2}^*+(-1)^i h_{3-i}\delta_{2n-4}^+}{h_1+h_2} +
    \dfrac{(-1)^{i+1}d_{2n-1}^* +h_{3-i} \ \delta_{2n-3}^+}{h_1+h_2}\,\xi
    \\[3mm]
    \qquad\qquad\qquad\qquad\qquad\qquad + \
    \dfrac{\xi^{2n-2}}{(2n-2)!}\dfrac{d^{2n-2}\omega_2^{(i)}}{dx^{2n-2}}(0) + \sum\limits_{j=0}^{2n-4}
    \dfrac{\xi^j}{j!}\dfrac{\partial^j{u_{2n-2-j}^{(i)}}}{\partial x^j}(0, \eta),
    \\[4mm]
    \qquad\qquad\qquad\qquad\qquad\qquad\qquad\qquad\qquad\qquad\qquad\quad
    \quad
     (-1)^i\xi \to +\infty, \ \eta \in \Upsilon_i, \ \ i=1, 2,
 \end{array}
 \right.
\end{equation}
and ${N}_{2n-2}$ has asymptotics (\ref{inner_asympt}).

Knowing $\delta_{2n-2}^+$ (see (\ref{const_d_0})) and using relations (\ref{solutions}), we get the solution $\omega_{2n}$ of problem (\ref{omega_probl*}).
Next coefficient ${N}_{2n-1}$ of the inner asymptotic expansion (\ref{junc}) is defined as the unique solution to problem~(\ref{junc_probl_n}) that can be rewritten in the form
\begin{equation}\label{junc_probl_odd_2n-2}
 \left\{
 \begin{array}{l}
 \begin{array}{rcll}
  -\Delta_{\xi\eta}{{N}_{2n-1}}(\xi, \eta)                                  & =   &
   \dfrac{\xi^{2n-3}}{(2n-3)!}\dfrac{\partial^{2n-3}f}{\partial x^{2n-3}}(0, \eta), &
   \quad (\xi, \eta)\in\Xi,
   \\[4mm]
  \partial_\nu{{N}_{2n-1}}(\xi, \eta)                             & =   & 0, &
   \quad (\xi, \eta)\in\Gamma,
   \\
  -\partial_\eta{{N}_{2n-1}}(\xi, \eta)|_{\eta=\pm \frac{h_i}{2}} & =   &
   \dfrac{\xi^{2n-3}}{(2n-3)!}\dfrac{d^{2n-3}\varphi_\pm^{(i)}}{dx^{2n-3}}(0), &
   \quad (-1)^i \xi \in (\tfrac{l}{2}, +\infty), \ \ i=1, 2,
 \end{array}
 \\[4mm]
 \qquad \qquad \qquad \quad
  {N}_{2n-1} \ \ \sim \ \,
    \dfrac{d_{2n+1}^*+(-1)^i h_{3-i}\delta_{2n-1}^+}{h_1+h_2} +
    \dfrac{(-1)^{i+1}d_{2n}^* +h_{3-i} \ \delta_{2n-2}^+}{h_1+h_2} \, \xi
    \\[3mm]
    \qquad\qquad\qquad\qquad\qquad\qquad + \
    \dfrac{\xi^{2n-1}}{(2n-1)!}\dfrac{d^{2n-1}\omega_2^{(i)}}{dx^{2n-1}}(0) + \sum\limits_{j=0}^{2n-3}
    \dfrac{\xi^j}{j!}\dfrac{\partial^j{u_{2n-1-j}^{(i)}}}{\partial x^j}(0, \eta),
    \\[4mm]
    \qquad\qquad\qquad\qquad\qquad\qquad\qquad\qquad\qquad\qquad\qquad\quad
    \quad
     (-1)^i\xi \to +\infty, \ \eta \in \Upsilon_i, \ \ i=1, 2,
 \end{array}
 \right.
\end{equation}

Coefficients $u_{2n}^{(i)}, \ i=1, 2,$ are determined as solutions of the following problems:
\begin{equation}\label{new_regul_probl_k}
\left\{\begin{array}{rcl}
-\partial_{\eta\eta}^2{u}_{2n}^{(i)}(x,\eta)    & = &  \partial_{xx}^2{u}_{2n-2}^{(i)}(x,\eta), \quad \eta\in\Upsilon_i, \\[2mm]
-\partial_{\eta}u_{2n}^{(i)}(x,\eta)|_{\eta=\pm\frac{h_i}{2}} & = & 0,     \quad x\in I_i,                       \\[2mm]
\langle u_{2n}^{(i)}(x,\cdot) \rangle_{\Upsilon_i}   & = & 0, \quad x\in I_i.
\end{array}\right.
\end{equation}
We note that solvability condition for problems (\ref{new_regul_probl_k}) takes place, because
$\langle u_{2n-2}^{(i)}(x,\cdot) \rangle_{\Upsilon_i} = 0,$ $i=1,2.$

Finally, we find the coefficients $\Pi_{2n}^{(1)}$ and $\Pi_{2n}^{(2)}$ of the boundary asymptotic expansions (\ref{prim+}) and (\ref{prim-}) respectively as solutions of problems (\ref{prim+probl})  and (\ref{prim-probl}) that can be rewritten in the form
\begin{equation}\label{new_prim+probl}
 \left\{\begin{array}{rll}
  -\Delta_{\xi\eta}\Pi_{2n}^{(1)}(\xi,\eta)=                      & 0,                  & (\xi,\eta)\in(0,+\infty)\times\Upsilon_1,                                     \\[2mm]
  -\partial_\eta\Pi_{2n}^{(1)}(\xi,\eta)|_{\eta=\pm\frac{h_1}{2}}=& 0,                  & \xi\in(0,+\infty),                                                            \\[2mm]
  \Pi_{2n}^{(1)}(0,\eta)=                                         & -u_{2n}^{(1)}(-1,\eta), & \eta\in\Upsilon_1,                                                            \\[2mm]
  \Pi_{2n}^{(1)}(\xi,\eta)\to                                     & 0,                  & \xi\to+\infty,\ \ \eta\in\Upsilon_1,
 \end{array}\right.
\end{equation}
\begin{equation}\label{new_prim-probl}
 \left\{\begin{array}{rll}
  -\Delta_{\xi^*\eta}\Pi_{2n}^{(2)}(\xi^*,\eta)=                    & 0,                  & (\xi^*,\eta)\in(0,+\infty)\times\Upsilon_2,                                     \\[2mm]
  -\partial_\eta\Pi_{2n}^{(2)}(\xi^*,\eta)|_{\eta=\pm\frac{h_2}{2}}=& 0,                  & \xi^*\in(0,+\infty),                                                           \\[2mm]
  \Pi_{2n}^{(1)}(0,\eta)=  & -u_{2n}^{(2)}(1,\eta), & \eta\in\Upsilon_2,                                                              \\[2mm]
  \Pi_{2n}^{(1)}(\xi^*,\eta)\to                                     & 0,                  & \xi^*\to+\infty, \ \ \eta\in\Upsilon_2.
 \end{array}\right.
\end{equation}

Thus we  successively determine all coefficients of series (\ref{regul}), (\ref{prim+}),  (\ref{prim-}) and (\ref{junc}).

\subsection{Justification}

Let us introduce the following notations
$$
u_k\big(x,\frac{y}{\varepsilon}\big)=
\left\{\begin{array}{rl}
u_k^{(1)}\left(x,\frac{y}{\varepsilon}\right), & x<0 \\[2mm]
u_k^{(2)}\left(x,\frac{y}{\varepsilon}\right), & x>0
\end{array}\right. ,
\quad
\omega_k(x)=
\left\{\begin{array}{rl}
\omega_k^{(1)}(x), & x<0 \\[2mm]
\omega_k^{(2)}(x), & x>0
\end{array}\right. , \quad
k \in \Bbb N, \ k \ge 2
$$
and define the coefficients of regular asymptotics as follows:
$$
u_k^*\left(x,\frac{y}{\varepsilon}\right) = u_{k}\left(x,\frac{y}{\varepsilon}\right) + \omega_{k+2}(x),
\quad
k \in \Bbb N_0
\quad
( u_0\equiv u_1 \equiv 0 ).
$$

With the help of the series (\ref{regul}), (\ref{prim+}),  (\ref{prim-}), (\ref{junc}) we construct the following series
\begin{multline}\label{asymp_expansion}
\sum\limits_{k=0}^{+\infty} \varepsilon^{k}\Bigg( \left(1 - \chi_l\left(\frac{x}{\varepsilon^\alpha}\right)\right) u_{k}^*\left(x,\frac{y}{\varepsilon}\right) + \chi_l \left(\frac{x}{\varepsilon^\alpha}\right) {N}_{k}\left(\frac{x}{\varepsilon},\frac{y}{\varepsilon}\right)\Bigg) 
\\
 + \sum\limits_{k=1}^{+\infty} \varepsilon^{2k}\Bigg(
  \chi^-(x)\Pi^{(1)}_{2k}\Big(\frac{1+x}{\varepsilon},\frac{y}{\varepsilon}\Big) + \chi^+(x)\Pi^{(2)}_{2k}\Big(\frac{1-x}{\varepsilon},\frac{y}{\varepsilon}\Big) \Bigg), \quad (x,y)\in\Omega_\varepsilon,
\end{multline}
where $\alpha$ is a fixed number from the interval $(\frac23, 1),$ $\chi_l, \ \chi^\pm$ are smooth cut-off functions defined by formulas
$$
\chi_l(x)=
\left\{\begin{array}{ll}
1, & \text{if} \ \ |x| < l,
\\
0, & \text{if} \ \ |x|  > 2 \, l,
\end{array}\right.
\quad
\chi^\pm(x)=
\left\{\begin{array}{ll}
1, & \text{if} \ \ |1\mp x| \le \delta,
\\
0, & \text{if} \ \ |1\mp x| \ge 2 \delta,
\end{array}\right.
$$
and $\delta$ is a sufficiently small fixed positive number.

\bigskip

\begin{theorem}\label{mainTheorem}
 Series $(\ref{asymp_expansion})$ is the asymptotic expansion for the solution of the boundary-value problem $(\ref{probl})$
 in the Sobolev space $H^1(\Omega_\varepsilon),$ i.e.,
\begin{equation}\label{t0}
    \forall \, m \in\Bbb{N} \ \ \exists \, {C}_m >0 \ \ \exists \, \varepsilon_0>0 \ \ \forall\, \varepsilon\in(0, \varepsilon_0) : \qquad \|u_\varepsilon-U_\varepsilon^{(m)}\|_{H^1(\Omega_\varepsilon)} \leq {C}_m \ \varepsilon^{\alpha (2\, m-\frac{1}{2}) + \frac{1}{2} },
\end{equation}
 where
\begin{multline}\label{aaN}
U^{(m)}_{\varepsilon}(x,y) =
\sum\limits_{k=0}^{2m} \varepsilon^{k}\Bigg( \left(1 - \chi_l\left(\frac{x}{\varepsilon^\alpha}\right)\right) u_{k}^*\left(x,\frac{y}{\varepsilon}\right) + \chi_l \left(\frac{x}{\varepsilon^\alpha}\right) {N}_{k}\left(\frac{x}{\varepsilon},\frac{y}{\varepsilon}\right)\Bigg) 
\\
 + \sum\limits_{k=1}^{m} \varepsilon^{2k}\Bigg(
  \chi^-(x)\Pi^{(1)}_{2k}\Big(\frac{1+x}{\varepsilon},\frac{y}{\varepsilon}\Big) + \chi^+(x)\Pi^{(2)}_{2k}\Big(\frac{1-x}{\varepsilon},\frac{y}{\varepsilon}\Big) \Bigg), \quad (x,y)\in\Omega_\varepsilon,
\end{multline}
is the partial sum of $(\ref{asymp_expansion}).$
\end{theorem}

\begin{remark}
Hereinafter, all constants in inequalities are independent of the parameter~$\varepsilon.$
\end{remark}
\begin{proof}
Take an arbitrary $m\in\Bbb{N}$. Substituting the partial sum $ U^{(m)}_{\varepsilon}$
in the equations and the boundary conditions of problem~(\ref{probl}) and taking into account  relations (\ref{main})--(\ref{new_prim-probl}) for the coefficients of series (\ref{asymp_expansion}), we find
\begin{equation}\label{t1}
\Delta U^{(m)}_{\varepsilon}(x,y) + f\left(x,\frac{y}{\varepsilon}\right) =
 \sum\limits_{j=1}^{6} R^{(m)}_{\varepsilon, j}(x,y) =: R^{(m)}_\varepsilon(x,y).
\end{equation}
where
\begin{equation}\label{t1_1}
R^{(m)}_{\varepsilon, 1}(x,y) = \varepsilon^{2m} \left(1 - \chi_l\left(\frac{x}{\varepsilon^\alpha}\right)\right) \dfrac{d^{2} u_{2m}}{dx^2}\left(x,\frac{y}{\varepsilon}\right),
\end{equation}
\begin{multline}\label{t1_2}
R^{(m)}_{\varepsilon, 2}(x,y) = \sum\limits_{k=1}^{2m}\varepsilon^k \Bigg(
2 \varepsilon^{-1-\alpha} \frac{d\chi_l}{d\zeta}(\zeta) \Big( \partial_{\xi}{N}_{k}(\xi,\eta) - \partial_{\xi}G_{k}(\xi,\eta) \Big)
\\
+
\varepsilon^{-2\alpha} \frac{d^2\chi_l}{d\zeta^2}(\zeta) \Big( {N}_{k}(\xi,\eta) - G_{k}(\xi,\eta) \Big)
\Bigg) \Bigg|_{\zeta=\frac{x}{\varepsilon^{\alpha}},\, \xi=\frac{x}{\varepsilon},\, \eta=\frac{y}{\varepsilon}}
\end{multline}
\begin{multline}\label{t1_3}
R^{(m)}_{\varepsilon, 3}(x,y) = \sum\limits_{k=1}^{m} \varepsilon^{2k} \Bigg( \Big( 2 \varepsilon^{-1}\frac{d\chi^-}{dx}(x) \partial_{\xi}\Pi^{(1)}_{2k}(\xi,\eta) + \frac{d^2\chi^-}{dx^2}(x)\Pi^{(1)}_{2k}(\xi,\eta) \Big) \Big|_{\xi=\frac{1+x}{\varepsilon},\, \eta=\frac{y}{\varepsilon}}
\\
+ \Big( 2 \varepsilon^{-1}\frac{d\chi^+}{dx}(x)\partial_{\xi}\Pi^{(2)}_{2k}(\xi,\eta) + \frac{d^2\chi^+}{dx^2}(x)\Pi^{(2)}_{2k}(\xi,\eta) \Big) \Big|_{\xi=\frac{1-x}{\varepsilon},\, \eta=\frac{y}{\varepsilon}} \Bigg),
\end{multline}
\begin{equation}\label{t1_4}
R^{(m)}_{\varepsilon, 4}(x,y) = \varepsilon^{\alpha(2m-1)} \chi_l\left(\frac{x}{\varepsilon^\alpha}\right) \frac{1}{(2m-2)!} \ \varepsilon^{-\alpha}  \int\limits_0^x \left(\frac{x-z}{\varepsilon^\alpha}\right)^{2m-2} \dfrac{\partial^{2m-1}f}{\partial z^{2m-1}}\left(z,\frac{y}{\varepsilon}\right) dz,
\end{equation}
\begin{multline}\label{t1_5}
R^{(m)}_{\varepsilon, 5}(x,y) = \varepsilon^{\alpha(2m-1)}  \Bigg( \varepsilon^{(1-\alpha)2m} \left( - \frac{d \omega_{2m+2}}{dx}(0) \ - \frac{\partial u_{2m}}{\partial x}\left(x,\frac{y}{\varepsilon}\right) \right)
\\
-
\frac{1}{(2m-1)!} \ \varepsilon^{-\alpha} \int\limits_0^x \left(\frac{x-z}{\varepsilon^\alpha}\right)^{2m-1} \dfrac{d^{2m+1}\omega_2}{d z^{2m+1}}(z) \ dz
\\
-
\sum\limits_{j=1}^{m-1} \frac{\varepsilon^{(1-\alpha)2j}}{(2m-2j-1)!} \ \varepsilon^{-\alpha} \int\limits_0^{x} \left(\frac{x-z}{\varepsilon^\alpha}\right)^{2m-2j-1} \frac{\partial^{2m-2j+1}u_{2j}}{\partial z^{2m-2j+1}}\left(z,\frac{y}{\varepsilon}\right) \ dz \Bigg)\cdot 2 \frac{d \chi_l}{d\zeta} (\zeta)|_{\zeta=\frac{x}{\varepsilon^\alpha}},
\end{multline}
\begin{multline}\label{t1_6}
R^{(m)}_{\varepsilon, 6}(x,y) = \varepsilon^{\alpha(2m-1)}
\Bigg( - \varepsilon^{(1-\alpha)2m-\alpha} x \frac{d \omega_{2m+2}}{dx}(0)
-
\frac{1}{(2m)!} \ \varepsilon^{-\alpha} \int\limits_0^x \left(\frac{x-z}{\varepsilon^\alpha}\right)^{2m} \dfrac{d^{2m+1}\omega_2}{d z^{2m+1}}(z) \ dz
\\
-
\sum\limits_{j=1}^{m} \frac{\varepsilon^{(1-\alpha)2j}}{(2m-2j)!} \ \varepsilon^{-\alpha} \int\limits_0^x \left(\dfrac{x-z}{\varepsilon^\alpha}\right)^{2m-2j} \dfrac{\partial^{2m-2j+1} u_{2j}}{\partial z^{2m-2j+1}}\left(z,\frac{y}{\varepsilon}\right) \ dz \Bigg) \cdot \frac{d^2 \chi_l}{d\zeta^2} (\zeta)|_{\zeta=\frac{x}{\varepsilon^\alpha}}.
\end{multline}

From (\ref{t1_1}) we conclude that
\begin{equation}\label{t2_1}
\exists\, \check{C}_m>0  \ \ \exists \, \varepsilon_0>0 \ \ \forall\, \varepsilon\in(0, \varepsilon_0) : \quad
\sup\limits_{(x,y)\in\Omega_{\varepsilon}}\left|R_{\varepsilon, 1}^{(m)}(x,y)\right| \leq \check{C}_m \varepsilon^{2m}.
\end{equation}
Due to the exponential decreasing of functions $\{{N}_{k} - {G}_{k},  \Pi^{(1)}_{k}, \Pi^{(2)}_{k}\}$ (see Remark~\ref{rem_exp-decrease} and (\ref{as_estimates})) and the fact that the support of the derivatives of cut-off function $\chi_l$ belongs to the set $\{x: l \varepsilon^\alpha \le |x| \le 2 l \varepsilon^\alpha\},$ we arrive that
\begin{equation}\label{t2_2}
\sup\limits_{(x,y)\in\Omega_{\varepsilon}}\left|R_{\varepsilon, 2}^{(m)}(x,y)\right| \leq \check{C}_m \varepsilon^{-1-\alpha} \exp{\left(-\frac{\pi l}{\max{(h_1, h_2)}\ \varepsilon^{1-\alpha}}\right)},
\end{equation}
similarly we obtain that
\begin{equation}\label{t2_3}
\sup\limits_{(x,y)\in\Omega_{\varepsilon}}\left|R_{\varepsilon, 3}^{(m)}(x,y)\right| \leq \check{C}_m \varepsilon^{-1} \exp{\left(-\frac{\pi \delta}{\max{(h_1, h_2)}\ \varepsilon}\right)}.
\end{equation}
We calculate terms $R_{\varepsilon, j}^{(m)}, \quad j=4,5,6$ with the help of the Taylor formula with the integral remaining term for functions $f$, $\omega_2$ and $\{u_{2k}\}$ at the point $x=0$. It is easy to check that
\begin{equation}\label{t2_456}
\sup\limits_{(x,y)\in\Omega_{\varepsilon}}\left| R_{\varepsilon, j}^{(m)}(x,y)\right| \leq \check{C}_m\varepsilon^{\alpha(2m-1)},
\quad
j=4,5,6.
\end{equation}

The partial sum leaves the following residuals in the boundary conditions:
$$
\begin{array}{rcll}
\partial_y{U_\varepsilon^{(m)}}(x,y)|_{y=\pm{\varepsilon\frac{h_i}{2}}} + {\varepsilon\varphi_\pm^{(i)}}(x)   & = &
\breve{R}_{\varepsilon, (i)_\pm}^{(m)}(x) , & x\in{I_\varepsilon^{(i)}}, \quad i=1,2,
\\[2mm]
U_\varepsilon^{(m)}(\pm1,y)  & = &
0, & y\in\Upsilon_\varepsilon^{(i)}, \quad i=1,2,
\\[2mm]
\partial_\nu{U_\varepsilon^{(m)}}(x,y)   & = &
0, & (x,y)\in\Gamma_\varepsilon,
\end{array}
$$
where
\begin{equation}\label{breve_R}
\breve{R}_{\varepsilon, (i)_\pm}^{(m)}(x) = \varepsilon^{1+\alpha(2m-1)} \chi_l\left(\frac{x}{\varepsilon^\alpha}\right) \frac{1}{(2m-2)!} \ \varepsilon^{-\alpha} \int\limits_0^x \left( \frac{x-z}{\varepsilon^\alpha} \right)^{2m-2} \frac{d^{2m-2}\varphi_\pm^{(i)}}{dz^{2m-2}}(z) \ dz, \quad i=1,2.
\end{equation}
It follows from (\ref{breve_R}) that there exist positive constants $\overline{C}_m$ and   $\overline{\varepsilon}_0$ such that
\begin{equation}\label{t2_7}
     \forall\, \varepsilon\in(0, \overline{\varepsilon}_0):\, \sup\limits_{x\in  I_\varepsilon^{(i)}} \left|\breve{R}_{\varepsilon,(i)_\pm}^{(m)}(x)\right| \leq \overline{C}_m\varepsilon^{1+\alpha(2m-1)}, \quad i=1,2.
\end{equation}

Using estimates (\ref{t2_1})~--~(\ref{t2_456}) and (\ref{t2_7}) we obtain the following estimates:
\begin{equation}\label{t3_1}
\left\|R_{\varepsilon, 1}^{(m)}\right\|_{L^2 (\Omega_\varepsilon)} \leq
\check{C}_m \sqrt{h_1+h_2} \ \varepsilon^{2m+\frac12},
\end{equation}
\begin{equation}\label{t3_2}
\left\|R_{\varepsilon, 2}^{(m)}\right\|_{L^2 (\Omega_\varepsilon)} \leq
\check{C}_m \sqrt{l \ \max{(h_1, h_2)} \, } \ \varepsilon^{-\frac{\alpha+1}{2}} \ \exp{\left(-\frac{\pi l}{\max{(h_1, h_2)} \ \varepsilon^{1-\alpha}}\right)},
\end{equation}
\begin{equation}\label{t3_3}
\left\|R_{\varepsilon, 3}^{(m)}\right\|_{L^2 (\Omega_\varepsilon)} \leq
\check{C}_m \sqrt{l \ \max{(h_1, h_2)} \, } \ \delta^{\frac{1}{2}} \varepsilon^{-\frac{1}{2}} \ \exp{\left(-\frac{\pi \delta}{\max{(h_1, h_2)} \ \varepsilon}\right)},
\end{equation}
\begin{equation}\label{t3_4}
\left\|R_{\varepsilon, 4}^{(m)}\right\|_{L^2 (\Omega_\varepsilon)} \leq
\check{C}_m \left({\frac32 l h_1 + \frac32 l h_2 + |\Xi^{(0)}|}\right)^\frac12 \ \varepsilon^{\alpha(2m-\frac12)+\frac12},
\end{equation}
\begin{equation}\label{t3_56}
\left\|R_{\varepsilon, j}^{(m)}\right\|_{L^2 (\Omega_\varepsilon)} \leq
\check{C}_m \sqrt{lh_1+lh_2} \ \varepsilon^{\alpha(2m-\frac12)+\frac12}, \quad j=5,6,
\end{equation}
\begin{equation}\label{t3_7}
\left\|\breve{R}_{\varepsilon,(i)_\pm}^{(m)}\right\|_{L^2 (I_\varepsilon^{(i)})} \leq
\overline{C}_m \sqrt{\frac32 l \, } \ \varepsilon^{1+\alpha(2m-\frac12)}, \quad i=1,2.
\end{equation}
Thus, the difference  $W_\varepsilon := u_\varepsilon - U_\varepsilon^{(m)}$ satisfies the following system:
\begin{equation}\label{nevyazka}
\left\{\begin{array}{rcll}
-\Delta W_\varepsilon                                       & = &
R_\varepsilon^{(m)} &
\mbox{in} \ \Omega_\varepsilon,
\\[2mm]
-\partial_y W_\varepsilon(x,\pm\varepsilon\frac{h_i}{2})    & = &
\breve{R}_{\varepsilon,(i)_\pm}^{(m)}(x), &
x\in I_\varepsilon^{(i)}, \, i=1,2,
\\[2mm]
W_\varepsilon(\pm 1,y)                                      & = &
0, &
y\in\Upsilon_\varepsilon^{(i)}, \ i=1,2,
\\[2mm]
\partial_\nu{W_\varepsilon}                                 & = &
0, &
\mbox{on} \ \Gamma_\varepsilon,
\end{array}\right.
\end{equation}
This means that the constructed series (\ref{asymp_expansion}) is a formal asymptotic solution of problem (\ref{probl}).

From (\ref{nevyazka}) we derive the following integral relation:
$$
 \int \limits_{\Omega_\varepsilon} {|\nabla W_\varepsilon |}^2 dxdy =
 \int \limits_{\Omega_\varepsilon} R_\varepsilon^{(m)} \, W_\varepsilon \,dx dy \mp
 \sum \limits_{i=1}^2 \ \int \limits_{I_\varepsilon^{(i)}}
 \breve{R}_{\varepsilon, (i)_\pm}^{(m)} \, W_\varepsilon |_{y=\pm\varepsilon\frac{h_i}{2}} \,dx.
$$
In view of the Friedrichs inequality  and estimates  (\ref{t3_1})~--~(\ref{t3_7}), this yields the following inequality:
$$
\int \limits_{\Omega_\varepsilon} {|\nabla W_\varepsilon |}^2 dxdy
\leq  \check{c}_m \ \varepsilon^{\alpha(2m-\frac12)+\frac12} \|W_\varepsilon\|_{L^2(\Omega_\varepsilon)} +
\overline{c}_m \ \varepsilon^{1+\alpha(2m-\frac12)} \sum\limits_{i=1}^2  \|W_\varepsilon (\cdot, \pm\varepsilon\tfrac{h_i}{2})\|_{L^2(I_\varepsilon^{(i)})}
$$
$$
\leq {C}_m \, \varepsilon^{\alpha(2m-\frac12)+\frac12} \|\nabla W_\varepsilon\|_{L^2(\Omega_\varepsilon)}.
$$
This, in turn, means the asymptotic estimate (\ref{t0}) and proves the theorem.
\end{proof}

\begin{corollary}\label{Corollary}
The difference between the solution $u_\varepsilon$ of problem (\ref{probl}) and the solution $\omega_2$ of the limit problem (\ref{main}) admits the following asymptotic estimate:
\begin{equation}\label{t5}
 \|u_\varepsilon - \omega_2\|_{H^1(\Omega_\varepsilon)} \leq {C}_0 \,
 \varepsilon.
\end{equation}

In thin rectangles $\Omega_{\varepsilon,\alpha}^{(i)}:=\big( I_{\varepsilon, \alpha}^{(i)}\times\Upsilon_\varepsilon^{(i)}\big),$ $i=1, 2,$ the following estimates hold:
\begin{equation}\label{t5+}
 \|u_\varepsilon - \omega_2 - \varepsilon \omega_3 \|_{H^1(\Omega_{\varepsilon,\alpha}^{(i)})} \leq {C}_1 \,
 \varepsilon^{\frac32}, \ \ i=1, 2;
\end{equation}
in addition,
\begin{equation}\label{t6}
 \|E^{(i)}_\varepsilon(u_\varepsilon) - \omega_2  \|_{H^1(I_{\varepsilon, \alpha}^{(i)})} \leq {C}_2 \,
 \varepsilon, \ \ i=1, 2,
\end{equation}
\begin{equation}\label{t7}
\max_{x\in \overline{I}_{\varepsilon, \alpha}^{(i)}} \left| E^{(i)}_\varepsilon(u_\varepsilon)(x) - \omega_2(x)
\right|  \leq {C}_3 \,  \varepsilon, \ \ i=1, 2,
\end{equation}
where $I_{\varepsilon, \alpha}^{(1)}:= (-1, - 2l \varepsilon^\alpha),$ $I_{\varepsilon, \alpha}^{(2)}:=(2l \varepsilon^\alpha, 1),$
$\alpha$ is a fixed number from the interval $(\frac23, 1),$ $\omega_3$ is defined by the formula $(\ref{solutions})$ and
$$
E^{(i)}_\varepsilon(u_\varepsilon)(x) = \frac{1}{\varepsilon\, h_i}\int_{\Upsilon^{(i)}_\varepsilon} u_\varepsilon(x,y)\, dy, \quad i=1, 2.
$$

In the  neighbourhood $\Omega^{(0)}_{\varepsilon, l}:= \Omega_\varepsilon\cap \{(x,y): \ x \in (-\varepsilon l, \varepsilon l)\}$ of the joint, we get estimates
 \begin{equation}\label{t-joint1}
 \|\nabla_{x y}u_\varepsilon  - \nabla_{\xi \eta} \, N_1\|_{L^2(\Omega^{(0)}_{\varepsilon, l})} \le
 \|u_\varepsilon - \omega_2(0) - \varepsilon \, N_1\|_{H^1(\Omega^{(0)}_{\varepsilon, l})} \leq {C}_4 \, \varepsilon^{\frac32 \alpha + \frac12}.
\end{equation}
\end{corollary}

\begin{proof} Denote by $\chi^{\varepsilon}_{l,\alpha}(\cdot):= \chi_l(\tfrac{\cdot}{\varepsilon^\alpha}).$
Using the smoothness of the functions $\{\omega_k\}_{k=2}^4$ and the exponential decay of  the functions
$\{ N_k-G_k \}_{k=1,2}$, $\Pi^{(1)}_{2}$ and $\Pi^{(2)}_{2}$  at infinity, we deduce the inequality
(\ref{t5}) from  estimate (\ref{t0}) at $m=1$:
$$
 \left\|u_\varepsilon-\omega_2\right\|_{H^1(\Omega_\varepsilon)}
 \leq \left\|u_\varepsilon-U^{(1)}_{\varepsilon}\right\|_{H^1(\Omega_\varepsilon)} +
 \Big\| - \chi^{\varepsilon}_{l,\alpha}\, \omega_2 + \chi^{\varepsilon}_{l,\alpha}\, N_0 + \varepsilon \left( \big(1 - \chi^{\varepsilon}_{l,\alpha}\big)\, \omega_3 + \chi^{\varepsilon}_{l,\alpha}\, N_1 \right)
$$
$$
  + \varepsilon^{2} \left( \big(1 - \chi^{\varepsilon}_{l,\alpha}\big)\,(u_{2} + \omega_{4}) +
  \chi^{\varepsilon}_{l,\alpha}\, {N}_{2} + \chi^- \Pi^{(1)}_{2} + \chi^+ \Pi^{(2)}_{2} \right) \Big\|_{H^1(\Omega_\varepsilon)}
$$
$$
 \leq
 \widetilde{C}_1 \, \varepsilon^{\frac32\alpha+\frac12}
  + \left\| \omega_2 - \omega_2(0) \right\|_{H^1(\Omega_\varepsilon^{(0)})}
 + \varepsilon \left\| N_1 \right\|_{H^1(\Omega_\varepsilon^{(0)})}
 + \varepsilon^2 \left\| N_2 \right\|_{H^1(\Omega_\varepsilon^{(0)})}
$$
$$
+ \varepsilon^2 \left(\| \chi^- \Pi_2^{(1)} \|_{H^1(\Omega_\varepsilon^{(1)})}
+ \| \chi^+ \Pi_2^{(2)} \|_{H^1(\Omega_\varepsilon^{(2)})} \right)
$$
$$
 + \sum\limits_{i=1}^2
  \Big\| \chi^{\varepsilon}_{l,\alpha} (\omega_2(0) - \omega_2) +  \varepsilon \left( (1-\chi^{\varepsilon}_{l,\alpha}) \omega_3 +
 \chi^{\varepsilon}_{l,\alpha} N_1 \right) +
 \varepsilon^{2} \left( (1-\chi^{\varepsilon}_{l,\alpha})(u_{2} + \omega_{4}) +
  \chi^{\varepsilon}_{l,\alpha} {N}_{2} \right) \Big\|_{H^1(\Omega^{(i)}_\varepsilon)}
$$
$$
\leq \widetilde{C}_1 \, \varepsilon^{\frac32\alpha+\frac12}
  + \varepsilon c_1 +  \varepsilon \left\| N_1 \right\|_{H^1(\Xi^{(0)})}
 + \varepsilon^2 \left\| N_2 \right\|_{H^1(\Xi^{(0)})}
+ \varepsilon^2 \left(\| \chi^- \Pi_2^{(1)} \|_{H^1(\Omega_\varepsilon^{(1)})}
+ \| \chi^+ \Pi_2^{(2)} \|_{H^1(\Omega_\varepsilon^{(2)})} \right)
$$
$$
+ \sum_{i=1}^2 \Big( \varepsilon \left\| \chi^{\varepsilon}_{l,\alpha}\,  ( N_1 - G_1 ) \right\|_{H^1(\Omega_\varepsilon^{(i)})}
 + \varepsilon^2 \left\| \chi^{\varepsilon}_{l,\alpha} \, ( N_2 - G_2 ) \right\|_{H^1(\Omega_\varepsilon^{(i)})} \Big)
$$
$$
+ \sum_{i=1}^2 \Big\| \chi^{\varepsilon}_{l,\alpha} \, \Big ( \omega_2(0) + x \frac{d\omega_2}{dx}(0)
 + \frac{x^2}{2} \frac{d^2\omega_2}{dx^2}(0) -  \omega_2 \Big) \Big\|_{H^1(\Omega_\varepsilon^{(i)})}
$$
$$
+ \varepsilon \sum_{i=1}^2  \Big\| \chi^{\varepsilon}_{l,\alpha}\, \Big( \omega_3(0) + x \frac{d\omega_3}{dx}(0)
 -  \omega_3 \Big) \Big\|_{H^1(\Omega_\varepsilon^{(i)})}
$$
$$
+\varepsilon^2  \sum_{i=1}^2 \left( \left\| \chi^{\varepsilon}_{l,\alpha}\, \big(u_2(0,\cdot) - u_2 \big) \right\|_{H^1(\Omega_\varepsilon^{(i)})}
 +  \left\| \chi^{\varepsilon}_{l,\alpha}\, \big(\omega_4(0) - \omega_4\big) \right\|_{H^1(\Omega_\varepsilon^{(i)})} \right)
$$
$$
+ \sum_{i=1}^2 \left( \varepsilon \left\| \omega_3 \right\|_{H^1(\Omega_\varepsilon^{(i)})} +
\varepsilon^2 \left\| u_2 + \omega_4 \right\|_{H^1(\Omega_\varepsilon^{(i)})} \right) \le C_1 \, \varepsilon.
$$

Again with the help of estimate (\ref{t0}) at $m=1,$ we deduce
$$
  \|u_\varepsilon - \omega_2 - \varepsilon \omega_3 \|_{H^1(\Omega_{\varepsilon,\alpha}^{(i)})} \leq
  \|u_\varepsilon-U^{(1)}_{\varepsilon}\|_{H^1(\Omega_\varepsilon)} +
\varepsilon^{2} \big\| u_{2} + \omega_{4} + \chi^- \Pi^{(1)}_{2} + \chi^+ \Pi^{(2)}_{2} \big\|_{H^1(\Omega_{\varepsilon,\alpha}^{(i)})}
$$
$$
\leq \widetilde{C}_1 \, \varepsilon^{\frac32\alpha+\frac12} + \widetilde{C}_2 \, \varepsilon^{\frac32},
$$
whence we get (\ref{t5+}).
Using the Cauchy-Buniakovskii-Schwarz inequality and (\ref{t5+}),  we obtain inequalities (\ref{t6}).
Since the space $H^1(I_{\varepsilon, \alpha}^{(i)})$ continuously embedded in $C(\overline{I}_{\varepsilon, \alpha}^{(i)}),$
 from (\ref{t6}) it follows inequalities (\ref{t7}).

From inequalities
$$
 \|u_\varepsilon - \omega_2(0) - \varepsilon \, N_1\|_{H^1(\Omega^{(0)}_{\varepsilon, l})} \leq
\| u_\varepsilon - U^{(1)}_{\varepsilon}\|_{H^1(\Omega_\varepsilon)} + \varepsilon^2 \|N_2\|_{H^1(\Omega^{(0)}_{\varepsilon, l})}
\leq \widetilde{C}_1 \, \varepsilon^{\frac32\alpha+\frac12} + \widetilde{C}_3 \, \varepsilon^2
$$
it follows more better energetic estimate (\ref{t-joint1}) in a neighbourhood of the joint $\Omega^{(0)}_\varepsilon.$
\end{proof}

\begin{remark}
If $\varphi^{(i)}_\pm \equiv 0$ and the function $f$ depends only on the variable $x,$ then all coeficient $\{u_{2k}\},$ $\{\Pi^{(1)}_{2k}\}$ and $\{\Pi^{(2)}_{2k}\}$ are equal to 0.
In this case the asymptotic series $(\ref{asymp_expansion})$ has the following form:
\begin{equation}\label{asymp_expansion1}
\sum\limits_{k=0}^{+\infty} \varepsilon^{k}\Bigg( \left(1 - \chi_l\left(\frac{x}{\varepsilon^\alpha}\right)\right) \omega_{k+2}(x) + \chi_l \left(\frac{x}{\varepsilon^\alpha}\right) {N}_{k}\left(\frac{x}{\varepsilon},\frac{y}{\varepsilon}\right)\Bigg), \quad (x,y)\in\Omega_\varepsilon,
\end{equation}
and the residual terms $\{R^{(m)}_{\varepsilon, j}\}_{j=1}^6$ are also simplified respectively, but the asymptotic estimates
$(\ref{aaN})$ remain the same. Nevertheless, as follows from the proof of Corollary~\ref{Corollary} the
asymptotic estimates $(\ref{t5+}) - (\ref{t7})$ become better:
\begin{equation}\label{t5++}
 \|u_\varepsilon - \omega_2 - \varepsilon \omega_3 \|_{H^1(\Omega_{\varepsilon,\alpha}^{(i)})} \leq {C}_1 \,
 \varepsilon^{\frac32 \alpha + \frac12}, \ \ i=1, 2;
\end{equation}
\begin{equation}\label{t6+}
 \|E^{(i)}_\varepsilon(u_\varepsilon) - \omega_2 - \varepsilon \omega_3 \|_{H^1(I_{\varepsilon, \alpha}^{(i)})} \leq {C}_2 \,
 \varepsilon^{\frac32 \alpha}, \ \ i=1, 2;
\end{equation}
\begin{equation}\label{t7+}
\max_{x\in \overline{I}_{\varepsilon, \alpha}^{(i)}} \left| E^{(i)}_\varepsilon(u_\varepsilon)(x) - \omega_2(x) - \varepsilon \omega_3(x)
\right|  \leq {C}_3 \,  \varepsilon^{\frac32 \alpha}, \ \ i=1, 2.
\end{equation}
\end{remark}

\section{Conclusions}

{\bf 1.} The energetic estimate (\ref{t5}) partly confirms the first formal result of \cite{Gaudiello-Kolpakov} (see p. 296)  that the local geometrical irregularity of  the analyzed structure  does not significantly affect on  the global-level properties of the framework, which are described by the limit problem (\ref{main}) and its solution $\omega_2$ (the leading term of the asymptotics).

But now, due to estimates (\ref{t5+}) and (\ref{t5++}) -- (\ref{t7+}) it became possible to identify the impact of
the geometric irregularity and material characteristics of the joint on the global level (the second term $\omega_3$ of the regular asymptotics (\ref{regul}) depends on the constant $\delta^+_{1}$ that takes into account all these factors (see (\ref{delta_1}))).
This our conclusion does not coincide with the second main result of \cite{Gaudiello-Kolpakov} (see p. 296) that
``{\it the joints of normal type manifest themselves on the local level only}''.

\smallskip

 In addition, in \cite{Gaudiello-Kolpakov} the authors stated that the main idea of their approach ``{\it
is to use a local perturbation corrector of the form $\varepsilon N({\bf x}/\varepsilon) \frac{d u_0}{dx_1}$ with the condition that the function $N({\bf y})$
is localized near the joint} '', i.e., $N({\bf y}) \to 0$ as $|{\bf y}| \to +\infty,$ and the main assumption of this approach is that $\nabla_{y} N \in L_1(Q_{\infty})$ (see (14) and similar assumptions on p. 300 and p. 303).

As we see the coefficients $\{N_k\}$ of the inner asymptotics (\ref{junc}) behave as polynomials at infinity and do not decrease exponentially (see (\ref{inner_asympt})).  Therefore, they influence directly the terms of the regular asymptotics beginning with the second one.
Thus, the main assumption made in \cite{Gaudiello-Kolpakov} is not satisfied.
This is the second our principal disparity with results of \cite{Gaudiello-Kolpakov}.

\medskip
\noindent
{\bf 2.}
From (\ref{t5}) it follows that the gradient $\nabla u_\varepsilon$ is equivalent to $\frac{d\omega_2}{dx}$ in the $L^2$-norm over whole junction $\Omega_\varepsilon$ as $\varepsilon \to 0.$
Since $\|\frac{d\omega_2}{dx}\|_{L^2(\Omega^{(0)}_{\varepsilon, l})} = {\cal O}(\varepsilon)$ as $\varepsilon \to 0,$ the estimate (\ref{t5}) is not informative in the neighbourhood $\Omega^{(0)}_{\varepsilon, l}$ of the joint $\Omega^{(0)}_{\varepsilon}.$

The form of the complete asymptotic expansion (\ref{asymp_expansion}) gives us possibility to improve the zero-order approximation
of the gradient (flux) of the solution both in the main parts~$I_{\varepsilon, \alpha}^{(i)},$ $i=1, 2,$ of the junction:
$$
\nabla u_\varepsilon(x, y)  \sim \frac{d\omega_2}{dx}(x) + \varepsilon \, \frac{d\omega_3}{dx}(x) \quad \text{as}\quad \varepsilon \to 0
$$
considering the geometric irregularity and material characteristics of the joint (see formulas (\ref{t5+}),  (\ref{t5++})), and
in the neighbourhood $\Omega^{(0)}_{\varepsilon, l}$ of the joint:
$$
\nabla u_\varepsilon(x, y) \sim \nabla_{\xi\eta}({N}_{1}(\xi, \eta))|_{\xi=\frac{x}{\varepsilon}, \eta= \frac{y}{\varepsilon}} \quad \text{as}\quad \varepsilon \to 0
$$
(see (\ref{t-joint1})). Also using estimates (\ref{t0}), we can obtain more better approximation of the solution and its gradient with preset accuracy.

\medskip
\noindent
{\bf 3.}
The results obtained give the right, in terms of practical application, to replace the complex boundary-value problem (\ref{probl}) with
the corresponding simpler $1$- dimensional boundary-value problem (\ref{main}) with sufficient accuracy that measured by the parameter $\varepsilon$ characterizing the thickness and the local geometrical irregularity. In this regard, the uniform pointwise estimates (\ref{t7}) and (\ref{t7+}), that are very important for applied problems, also confirm this conclusion.

\medskip
\noindent
{\bf 4.}
The method proposed in the present paper for the construction of asymptotic expansions can be used
for the asymptotic investigation of boundary-value problems in graph-junctions of thin domains (Fig. \ref{fig5}),
or graph-junctions of thin perforated domains with rapidly varying thickness. In the last case, it is necessary to add
series with rapidly oscillating coefficients to the regular part of the asymptotics (see \cite{Mel_Pop_MatSb}).
\begin{figure}[htbp]
\begin{center}
\includegraphics[width=10cm]{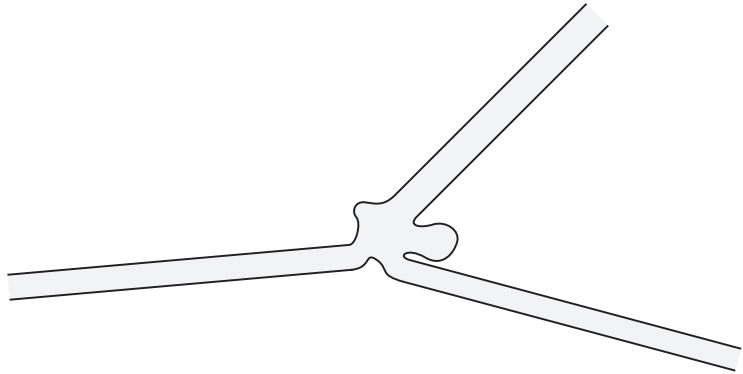}
\caption{A graph-junction of thin domains with a local joint}\label{fig5}
\end{center}
\end{figure}

\end{document}